\documentclass[a4paper,11pt,reqno]{amsart}%
\usepackage{amsfonts}
\usepackage{amsthm}
\usepackage{hyperref}
\usepackage{amssymb}
\usepackage{graphicx}
\usepackage{float}
\usepackage{cite}
\usepackage{eqparbox,array}
\usepackage{eucal,times,color,enumerate,accents}
\usepackage{url}
\usepackage{bbm}
\usepackage{microtype}
\usepackage{color}
\usepackage{tikz}
\usepackage{tikz-cd}
\usetikzlibrary{arrows}
\usepackage{pifont}
\newcommand{\cmark}{\ding{51}}
\newcommand{\xmark}{\ding{55}}

\usepackage{dcolumn}
\usepackage{amsmath}
\usepackage{verbatim}
\newcolumntype{r}{D{.}{.}{-1}}

\hypersetup{
    colorlinks,
    linkcolor={red!70!black},
    citecolor={green!60!black},
    urlcolor={blue!80!black}
}

\newtheorem{theorem}{Theorem}[section]

\theoremstyle{plain}

\newtheorem{corollary}[theorem]{Corollary}

\theoremstyle{definition}
\newtheorem{definition}[theorem]{Definition}
\newtheorem{example}[theorem]{Example}
\newtheorem{construction}[theorem]{Construction}

\theoremstyle{plain}
\newtheorem{lemma}[theorem]{Lemma}

\newtheorem{proposition}[theorem]{Proposition}

\theoremstyle{remark}
\newtheorem{remark}[theorem]{Remark}
\numberwithin{equation}{section}

\DeclareMathOperator{\trop}{trop}
\DeclareMathOperator{\pearl}{pearl}
\DeclareMathOperator{\ev}{ev}
\DeclareMathOperator{\val}{val}

\DeclareMathOperator{\mult}{mult}
\DeclareMathOperator{\Aut}{Aut}
\DeclareMathOperator{\Coef}{Coef}

\newcommand {\PP}{{\mathbb P}}
\newcommand {\RR}{{\mathbb R}}
\newcommand {\ZZ}{{\mathbb Z}}

\newcommand {\NN}{{\mathbb N}}
\newcommand {\TT}{{\mathbb T}}

\begin{document}
\title[Counts of (tropical) curves in $E\times \mathbb{P}^1$ and Feynman integrals]{Counts of (tropical) curves in $E\times \mathbb{P}^1$ and Feynman integrals}
\author[Janko B\"{o}hm, Christoph Goldner, Hannah Markwig]{Janko B\"{o}hm, Christoph Goldner, Hannah Markwig}
\address{Janko B\"ohm, Fachbereich Mathematik,
Universit\"{a}t Kaiserslautern, Postfach 3049, 67653 Kaiserslautern, Germany}
\email{\href{mailto:boehm@mathematik.uni-kl.de}{boehm@mathematik.uni-kl.de}}
\address{Hannah Markwig, Eberhard Karls Universit\"at T\"ubingen, Fachbereich Mathematik}
\email{\href{mailto:hannah@math.uni-tuebingen.de}{hannah@math.uni-tuebingen.de}}
\address{Christoph Goldner, Eberhard Karls Universit\"at T\"ubingen, Fachbereich Mathematik}
\email{\href{mailto:christoph.goldner@math.uni-tuebingen.de}{christoph.goldner@math.uni-tuebingen.de}}

\thanks{2010 Math subject classification: Primary 14N35, 14T05, 81T18; Secondary 11F11, 14J27, 14N10, 14J81}

\begin{abstract}
We study generating series of Gromov-Witten invariants of $E\times\PP^1$ and their tropical counterparts. Using tropical degeneration and floor diagram techniques, we can express the generating series as sums of Feynman integrals, where each summand 
corresponds to a certain type of graph which we call a \emph{pearl chain}. The individual summands are --- just as in the case of mirror symmetry of elliptic curves, where the generating series of Hurwitz numbers equals a sum of Feynman integrals --- complex analytic path integrals involving a product of propagators (equal to the Weierstrass-$\wp$-function plus an Eisenstein series). 
We also use pearl chains to study generating functions of counts of tropical curves in $E_\TT\times\PP^1_\TT$ of so-called \textit{leaky degree}.
\end{abstract}
\maketitle

\section{Introduction}

\subsection{Generating series of Grommov-Witten invariants of $E\times\PP^1$ and Feynman integrals}
We study generating series of Gromov-Witten invariants of $E\times\PP^1$. These can be viewed as counts of curves in $E\times\PP^1$ of fixed bidegree $(d_1,d_2)$ and genus $g$ and satisfying generic point conditions. We denote such a number by $N_{(d_1,d_2,g)}$. 

Our main result states that, for fixed $d_2$ and $g$, the generating series $\sum_{d_1}N_{(d_1,d_2,g)} q^{d_1}$ equals a sum of Feynman integrals (see Corollary \ref{cor-genseries=Feynman,class}):

\begin{equation}\sum_{d_1}N_{(d_1,d_2,g)} q^{d_1} = \sum_{\mathcal{P}}\frac{1}{|\Aut(\mathcal{P})|} I_{\mathcal{P}}(q).\label{eq-mainresult}
\end{equation}

A Feynman integral $I_{\mathcal{P}}(q)$ can be viewed as a path integral of a product of propagator functions involving the Weierstra\ss{}-$\wp$-function and an Eisenstein series in a Cartesian product of elliptic curves. Alternatively, it can be viewed as the constant coefficient of a series involving the analogous product of propagators after a coordinate change. The way the product of propagators is given depends on a graph $\mathcal{P}$. 

In Equation (\ref{eq-mainresult}), we have to sum over particular graphs $\mathcal{P}$ which we call \emph{pearl chains} (see Definition \ref{def-pearl}). For fixed $d_2$ and $g$, there is a finite list of pearl chains of type $(d_2,g)$.

Our study is inspired by Dijkgraaf's famous mirror symmetry theorem for elliptic curves relating generating series of Hurwitz numbers and Feynman integrals \cite{Dij95}. A Hurwitz number is a count of simply ramified covers of an elliptic curve of fixed genus $g$ and degree $d$. We denote such a number by $N_{d,g}$.
The mirror symmetry theorem for elliptic curves states that, for fixed $g\geq 2$:

\begin{equation} \sum_d N_{d,g}q^d  = \sum_{\Gamma} \frac{1}{|\Aut(\Gamma)|} I_{\Gamma}(q). \label{eq-mirrorsymm} \end{equation}

Here, the sum on the right goes over all $3$-valent connected graphs $\Gamma$ of genus $g$.

It is interesting that our generating series of Gromov-Witten invariants of $E\times\PP^1$ can be expressed as a sum over the same kind of Feynman integrals, it is only the graphs over which we sum that changes when comparing Equation (\ref{eq-mirrorsymm}) for Hurwitz numbers to Equation (\ref{eq-mainresult}) for Gromov-Witten invariants of $E\times\PP^1$.

The mirror symmetry relation (\ref{eq-mirrorsymm}) was used in \cite{Dij95, KZ95} to prove that the  generating function of Hurwitz numbers is a quasimodular form of weight $6g-6$. Quasimodularity behaviour is desirable because it controls the asymptotic of the generating function. Recently, the quasimodularity result for generating series of Hurwitz numbers was generalized in \cite{OP17} to cycle-valued generating series involving arbitrary Gromov-Witten invariants of an elliptic curve (see Theorem 2 and Corollary 1 in \cite{OP17}). There, Oberdieck and Pixton also considered elliptic fibrations and conjecture a quasimodularity statement for cycle-valued generating series of Gromov-Witten invariants (Conjecture A \cite{OP17}), which they prove for the case of products $E\times B$ (Corollary 2).

Besides giving an explicit description for generating series of Gromov-Witten invariants of $E\times\PP^1$, our Equation (\ref{eq-mainresult}) also hands us a new way to study quasimodularity by making use of Feynman integrals: the quasimodularity of a summand $I_{\mathcal{P}}$ 
for a fixed graph $\mathcal{P}$ can be deduced from \cite{OP17} (see Theorem \ref{thm-quasimod homo}).

We rely on tropical geometry to prove Equation (\ref{eq-mainresult}). This method also gives an interpretation of a summand corresponding to a graph $\mathcal{P}$ on the left hand side of Equation (\ref{eq-mainresult}): it can be viewed as a generating series counting stable maps with a fixed underlying graph $\mathcal{P}$ close to the tropical limit.

\subsection{Tropical curve counts, floor diagrams and curled pearl chains}

We prove a corres-pondence theorem stating the equality of the Gromov-Witten invariant $N_{(d_1,d_2,g)}$ to its tropical counterpart (see Theorem \ref{thm-corres}).
Tropical geometry can be viewed as a degenerate version of algebraic geometry and has become an important tool in (log-) Gromov-Witten theory and curve counts (see e.g.\ \cite{Ran15, BG14, BBM11}), starting with Mikhalkin's breakthrough \cite{Mi03} where he proved the first correspondence theorem.

We then use floor diagram techniques to relate counts of tropical curves to counts of \emph{curled pearl chains} (see Definition \ref{def-curledpearl}) --- these can essentially be viewed as (combinatorial types of) tropical covers of a tropical elliptic curve with a particular source graph, namely a pearl chain.

Floor diagrams are a way to carve out the combinatorial essence of a tropical curve count. They can also be viewed as the graphs illustrating the ultimate use of the degeneration formula for Gromov-Witten invariants \cite{Li02, Bexp, AB14, CJMR17}.
They were introduced for counts of curves in $\PP^2$ by Brugall\'e{}-Mikhalkin \cite{BM:Pn}, and further investigated by Fomin-Mikhalkin \cite{FM09}, leading to new results about node polynomials. The results were generalized to other toric surfaces in \cite{AB13}.

We also study curves of \textit{leaky degree} (see Subsection \ref{subsec-generalization}), which will be useful for generalizations to curve counts in $E\times\mathbb{P}^1$ for which we impose tangency conditions relative to the $0$- and $\infty$-section. Our methods apply to these curves as well.

In \cite{BG14b, CJMR17}, the floor diagram technique is compared to the Fock space technique, in which Feynman diagrams for operators are used to provide generating series of counts of curves on surfaces. Being well-known for counts of covers \cite{OP06}, the Fock space approach for counts of curves on surfaces was pioneered by Cooper and Pandharipande in \cite{CP12}. By applying the techniques of \cite{BG14b, CJMR17, CJMR16, BGM18}, we believe that the trace formula (Theorem 3 of \cite{CP12}) for generating series of counts of curves in $E\times\mathbb{P}^1$ can be deduced from our formula \ref{eq-mainresult} and vice versa.

\subsection{Tropical mirror symmetry of an elliptic curve and beyond}

The well-known Gross-Siebert program for mirror symmetry aims at constructing new mirror pairs and providing an algebraic framework for SYZ-mirror symmetry \cite{GS06, GS07, SYZ}. For a pair of an algebraic variety $X$ and a mirror $X^\vee$, we can express invariants of one in terms of the other. In particular, we can hope to express a generating series of Gromov-Witten invariants in terms of integrals for the mirror.
The philosophy how tropical geometry can be exploited is illustrated in the following triangle:

\[\scalebox{0.90}{
\begin{tikzpicture}[<->,>=stealth',shorten >=1pt,auto]
\coordinate (a) at (0,0);
\coordinate (b) at (-3.4,3.5);
\coordinate (c) at (3.4,3.5);
\node(1) at (a)  {\begin{tabular}{c}tropical\\ GW-invariants\end{tabular}};
\node(2) at (b) {\begin{tabular}{c}Gromov-Witten\\ invariants\end{tabular}};
\node(3) at (c)  {\begin{tabular}{c}(Feynman)\\ integrals\end{tabular}};
\path[every node/.style={}]
(1) edge node[left,sloped, anchor=center] {\begin{tabular}{c}Correspondence\\ Theorem\end{tabular}} (2)
(2) edge [dashed] node {Mirror symmetry} (3)
(3) edge node[right,sloped, anchor=center] {} (1);
\end{tikzpicture}
}
\]

In many situations, correspondence theorems relating Gromov-Witten invariants to their tropical counterparts are known \cite{Mi03, NS06, BBM10, CJM10}. If we can relate the generating function of tropical invariants to integrals, we obtain a proof of the desired mirror symmetry relation using a detour via tropical geometry \cite{Gro09, Ove15}.

The first and last author together with Bringmann and Buchholz studied the triangle above for the case of Hurwitz numbers of an elliptic curve, revealing that indeed a relation between counts of tropical covers and Feynman integrals holds, and that it even holds on a fine level, i.e.\ summand by summand \cite{BBBM13}. The tropical mirror symmetry theorem in particular implies Equation (\ref{eq-mirrorsymm}).

In \cite{BGM18}, we generalized the tropical mirror symmetry theorem to involve arbitrary descendant Gromov-Witten invariants. 

The main ingredient to prove tropical mirror symmetry is a bijection between certain covers of graphs and monomial contributing to a Feynman integral (see Theorem 2.23 \cite{BGM18}). The bijection we prove in \cite{BBBM13} was tailored to the case of Hurwitz numbers and does not apply to other situations. In \cite{BGM18}, we distilled the capacity of the bijective method, giving a theorem which holds for general types of graph covers which can essentially be viewed as combinatorial types of leaky tropical covers.

To relate generating series of curled pearl chains to Feynman integrals (thus proving Equation \ref{eq-mainresult}), we invoke this general bijection studied in \cite{BGM18}.

Again, it is easy to treat the case of tropical curve counts of leaky degree with the same methods. On the Feynman integral side, we then have to consider arbitrary coefficients of the formal power series mentioned above and not just the constant coefficient. The interpretation as path integral in complex analysis is restricted to the case of constant coefficients.

\subsection{Overview of the results}
Our methodology and results can be summed up by the following chart:

Equation (\ref{eq-mainresult}) as stated above is the equality of the very left side with the very right side. For the intermediate equalities, we chose to prove more general versions (involving tropical curve counts of leaky degree, curled pearl chains with leaking and Feynman integrals which are (non-constant) coefficients of power series), since these generalizations can be obtained essentially with the same effort and have potentially further applications in the theory of tropical curve counts.

\begin{figure}[H]
\centering
\def\svgwidth{395pt}
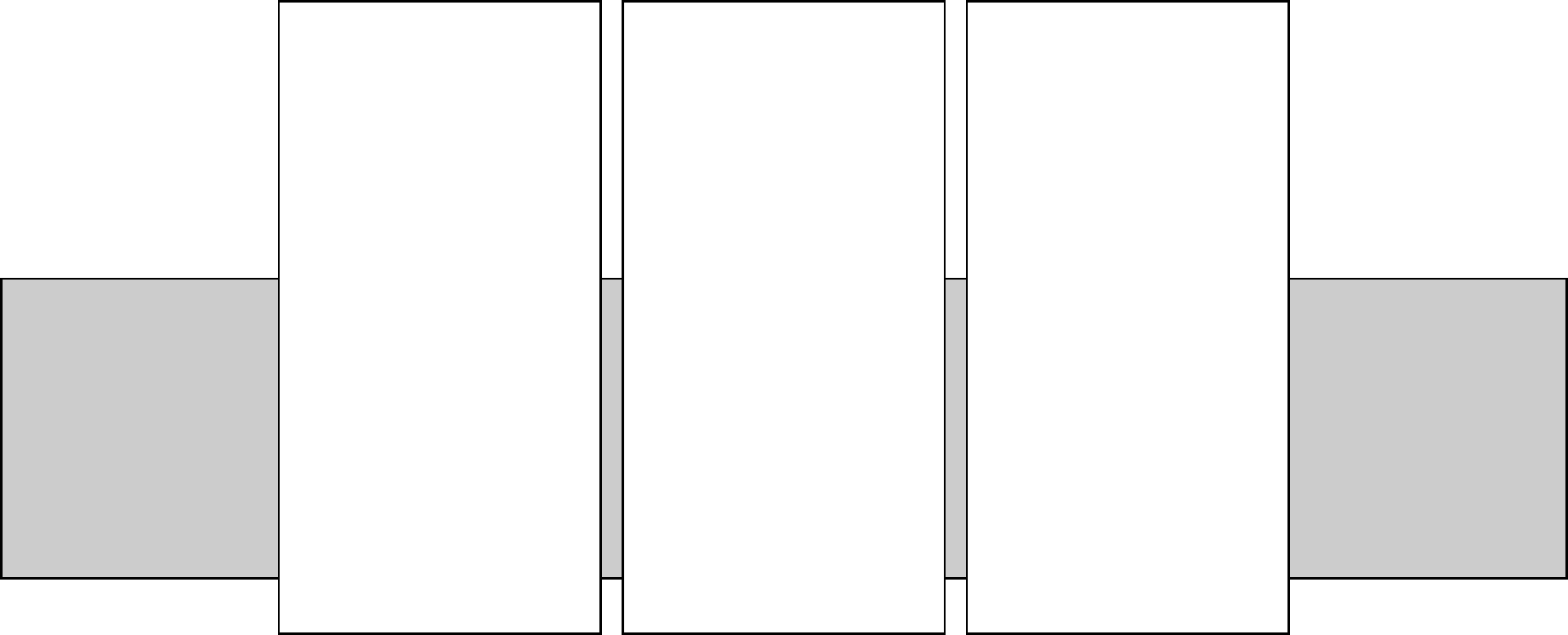
\label{fig-chart}
\end{figure}

\subsection{Further generalization and limitations}\label{subsec-generalization}

For our study of generating series of Gromov-Witten invariants of $E\times\PP^1$, we restrict to invariants evaluating point conditions. This puts us within the scope of the current techniques for correspondence theorems. It is also possible to evaluate points and insert Psi-conditions, i.e.\ study stationary descendant Gromov-Witten invariants. Preliminary work on correspondences exists for this case \cite{CJMR17, MR16}. It is more complicated to relate the generating series of descendant Gromov-Witten invariants to Feynman integrals however, as we can deduce from the experience in dimension one: In \cite{BGM18}, we study tropical mirror symmetry for descendant Gromov-Witten invariants of an elliptic curve (see also \cite{Lithesis, Li11}). The difficulty that arises can be expressed in terms of the multiplicity with which we need to count tropical covers: if we restrict to Hurwitz numbers, the multiplicity of a tropical cover is the product of edge expansion factors. In the case of descendant Gromov-Witten invariants, we also obtain local vertex contributions which are $1$-point relative descendant Gromov-Witten invariants. In dimension one, the generating series of $1$-point relative descendant Gromov-Witten invariants has a nice form and can be given in terms of sinus hyperbolicus \cite{OP}. This nice form enables us to single out vertex contributions in Feynman integrals and to prove a version of Equation (\ref{eq-mirrorsymm}) involving descendant Gromov-Witten invariants.

To count descendant Gromov-Witten invariants of $E\times\PP^1$ tropically, we also have local vertex contributions which are $1$-point relative descendant Gromov-Witten invariants, but now of $\PP^1\times\PP^1$. We are not aware of a nice form for the generating series of those. This momentarily limits our possibilities to generalize Equation (\ref{eq-mainresult}) to descendant Gromov-Witten invariants.

As we outline in Construction \ref{const-gluing}, our methods can also be viewed as a step towards future research involving counts of curves in $E\times\PP^1$ satisfying tangency conditions with the $\infty$-section.

\subsection{Organization of this paper}

Section \ref{sec-trop} is devoted to tropical curve counts and correspondence theorems. In Section \ref{sec-pearl}, we introduce pearl chains and curled pearl chains and prove the equality of counts of curled pearl chains to tropical curve counts of leaky degree. In Section \ref{sec-Feynman}, we prove the equality of our generating series to Feynman integrals.
Consequences concerning quasimodularity are discussed in Section \ref{sec-quasimod}.

\subsection{Acknowledgements}
We would like to thank Renzo Cavalieri, Georg Oberdieck, Dhruv Ranganathan and Kristin Shaw for helpful discussions. We thank an anonymous referee for pointing out a mistake in an earlier version.
The authors gratefully acknowledge support by DFG-collaborative research center TRR 195 (Project I.1 INST 248/237-1). Computations have been made with SINGULAR using the \textit{ellipticcovers} library. 
Part of this work was completed during the Mittag-Leffler programm \emph{Tropical geometry, amoebas and polytopes} in spring 2018. The authors would like to thank the institute for hospitality and excellent working conditions.

\section{Gromov-Witten invariants of $E\times \mathbb{P}^1$ and their tropical counterparts}\label{sec-trop}
\subsection{Gromov-Witten invariants of $E\times\PP^1$ and relative Gromov-Witten invariants of $\PP^1\times\PP^1$}
Gromov-Witten invariants are virtually enumerative intersection numbers on moduli spaces of stable maps. Let $E$ be an elliptic curve. We study Gromov-Witten invariants of $E\times \PP^1$. A  {\it stable map} of bidegree $(d_1,d_2)$ from a curve of genus $g$ to $E\times \PP^1$ with $n$ markings is a map $f: C \to E$, where  $C$ is a connected projective curve with at worst nodal singularities, and with $n$ distinct nonsingular marked points $x_1,\ldots,x_n\in C$, such that $f_\ast([C])$ is of class $(d_1,d_2)$ and $f$ has a finite group of automorphism. 
The moduli space of stable maps, denoted  $\overline{\mathcal{M}}_{g,n} (E\times\PP^1,(d_1,d_2))$, is a proper Deligne-Mumford stack of virtual dimension $2d_2+g-1+n$ \cite{Beh97,BF97}. 
The $i$th  evaluation morphism is the map $\ev_i: \overline{\mathcal{M}}_{g,n} (E\times \PP^1,(d_1,d_2)) \to E\times\PP^1$ that sends a point $[C, x_1, \ldots, x_n, f]$ to  $f(x_i) \in E\times \PP^1$. 

 \begin{definition}\label{def-GW}
 Fix $g,n,(d_1,d_2)$ with $n=2d_2+g-1$. The {\it Gromov-Witten invariant}  $\langle \tau_{0}(pt)^n  \rangle_{g,n}^{E\times\PP^1,(d_1,d_2)}$ is  defined as follows. As these numbers are the key players in this paper, we introduce the special notation $N_{(d_1,d_2,g)}$ as well:

 \begin{equation}
N_{(d_1,d_2,g)}= \langle \tau_{0}(pt)^n  \rangle_{g,n}^{E\times\PP^1,(d_1,d_2)} = \int_{[\overline{\mathcal{M}}_{g,n} (E\times\PP^1,(d_1,d_2))]^{vir}} \prod_{i=1}^n \ev_i^\ast(pt) 
 \end{equation}
 where $pt$ denotes the class of a point in $E\times \PP^1$.
 \end{definition}

To relate Gromov-Witten invariants of $E\times \PP^1$ to their tropical counterparts, we use a degeneration argument relating them to {\it relative Gromov-Witten invariants} of $\PP^1\times \PP^1$, relative to the $0$- and $\infty$-section.

Let $\mu^+$, $\phi^+$, $\mu^-$ and $\phi^-$ be partitions such that the sum $d_1$ of the parts in $\mu^+$ and $\phi^+$ equals the sum of the parts in $\mu^-$ and $\phi^-$. Let $n_1=\ell(\phi^+)+\ell(\phi^-)$ and $n_2=\ell(\mu^+)+\ell(\mu^-)$.
Consider the moduli space of {\it relative stable maps to $\PP^1\times\PP^1$} $$\overline{\mathcal{M}}_{g,n} (\PP^1\times\PP^1, (\mu^+,\phi^+),(\mu^-,\phi^-),(d_1,d_2)),$$ where part of the data specified are the partitions of contact orders $(\mu^+,\phi^+)$ resp.\ $(\mu^-,\phi^-)$ which we fix over the $0$- resp.\ $\infty$-section in $\PP^1\times\PP^1$. The points of contact with the $0$- and $\infty$-section are marked. 
We want to fix the points with contact orders given by $\phi^+$ and $\phi^-$, the ones with contact orders given by $\mu^+$ and $\mu^-$ are allowed to move.
A detailed discussion of spaces of relative stable maps and their boundary can be found e.g.\ in \cite{Vak08}.
This moduli space is a Deligne-Mumford stack of virtual dimension $(g-1) +2d_2 +n+n_1+n_2$. For $i=1,\ldots,n$, the $i$th  evaluation morphism is the map $\ev_i: \overline{\mathcal{M}}_{g,n} (\PP^1\times\PP^1, (\mu^+,\phi^+),(\mu^-,\phi^-),(d_1,d_2)) \to \PP^1\times\PP^1$ that sends a point $[C, x_1, \ldots, x_n, f]$ to  $f(x_i) \in \PP^1\times \PP^1$. 
The points marking the contact points with the $0$- and $\infty$-section give rise to evaluation morphisms 
\[
\widehat{\ev_i}: \overline{\mathcal{M}}_{g,n} (\PP^1\times\PP^1, (\mu^+,\phi^+),(\mu^-,\phi^-),(d_1,d_2))\to \PP^1.
\] 
Here, the target $\PP^1$ is the $0$-section for $\phi^-$ and $\mu^-$, and the $\infty$-section for $\phi^+$ and $\mu^+$.

\begin{definition}\label{def-relGWI}
The \emph{relative Gromov--Witten} invariant is defined as the following intersection number on $\overline{\mathcal{M}}_{g,n} (\PP^1\times\PP^1, (\mu^+,\phi^+),(\mu^-,\phi^-),(d_1,d_2))$:
\begin{equation}\label{eq-lgw}
\langle ({{\phi}}^-,{{\mu}}^-)| \tau_{0}(pt)^n|({{\phi}}^+,{{\mu}}^+)\rangle_{g,n}^{\PP^1\times\PP^1,(d_1,d_2)}= \int \prod_{j=1}^n \ev_j^\ast([pt])\prod_{i=n+1}^{n+n_1}\widehat{\ev}_i^\ast([pt])
\end{equation}
\end{definition}

One can allow source curves to be disconnected, and introduce {\it disconnected Gromov-Witten invariants}. We will add the superscript $\bullet$ anytime we refer to the disconnected theory.

The following statement is a consequence of the degeneration formula \cite{Li01A,Li02}, see also Theorem 4.7 in \cite{CJMR16}:
\begin{proposition}\label{prop-degeneration}
A Gromov-Witten invariant of $E\times\PP^1$ equals a weighted sum of relative Gromov-Witten invariants of $\PP^1\times\PP^1$:
$$N_{(d_1,d_2,g)}^\bullet= \sum_{(\mu,\phi)\; \vdash d_1} \frac{\prod_i \mu_i \prod_j\phi_j}{|\Aut(\mu)||\Aut(\phi)|}
\langle (\mu,\phi)|\tau_{0}(pt)^n |(\phi,\mu) \rangle_{g-\ell(\mu)-\ell(\phi), n}^{\PP^1\times\PP^1,(d_1,d_2),\bullet}.
$$
Here, the sum goes over all tuples of partitions which together form a partition $(\mu,\phi)$ of $d_1$.
\end{proposition}

\subsection{Tropical curves}
An \emph{abstract tropical} \emph{curve} is a  metric graph $\Gamma$ with unbounded edges called \emph{ends} which have infinite length. We only consider explicit tropical curves here. Locally around a point $p$, $\Gamma$ is homeomorphic to a star with $r$ halfrays. 
The number $r$ is called the \emph{valence} of the point $p$ and denoted by $\val(p)$. We identify the vertex set of $\Gamma$ as the points of valence different from $2$. 
Edges adjacent to leaf vertices (i.e.\ vertices of valence $1$) are required to have infinite length and are also called \emph{ends}.
Vertices of valence greater than $1$ are called  \textit{inner vertices}. Besides \emph{edges}, we introduce the notion of \emph{flags} of $\Gamma$. A flag is a pair $(V,e)$ of a vertex $V$ and an edge $e$ incident to it ($V\in \partial e$). Edges that are not ends are required to have finite length and are referred to as \emph{bounded} or \textit{internal} edges.
A \emph{marked tropical curve} is a tropical curve whose ends are (partially) labeled. An isomorphism of a marked tropical curve is an isometry respecting the end markings. The \emph{genus} of $\Gamma$ is given by its first Betti number.
The \emph{combinatorial type} is the equivalence class of tropical curves obtained by identifying any two tropical curves which differ only by edge lengths.

\subsection{Tropical $E\times \mathbb{P}^1$}

We denote the tropical numbers $\RR\cup \{- \infty\}$ by $\TT$.
The tropical projective line, $\PP^1_{\TT}$, equals $\RR\cup \{\pm \infty\}$. As in algebraic geometry, it is glued from two copies of the affine line $\TT$ using the tropicalization of the identification map on $\RR$: $x\mapsto -x$.

A (nondegenerate) tropical elliptic curve $E_{\TT}$ is a circle with a fixed length. 
The tropical analogue of the surface we are interested in here, $E_\TT\times\PP^1_\TT$ can be viewed as an infinite cylinder. It is a tropical surface in the sense of Definition 3.1 \cite{Shaw15}.

\begin{figure}[H]
\centering
\def\svgwidth{150pt}
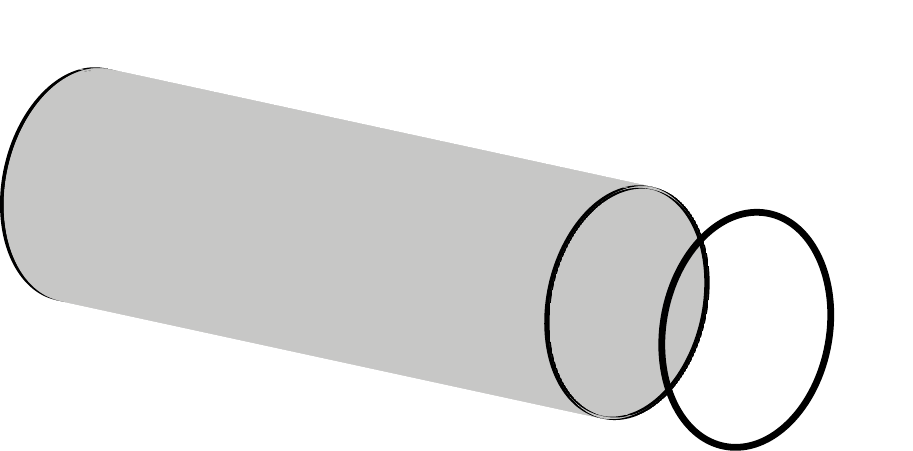
\caption{An illustration of $E_\TT\times\PP^1_\TT$.}
\label{fig-Pic2_ink}
\end{figure}

\subsection{Tropical stable maps}
A nice feature of tropical geometry is that often, we do not have to use compactifications to get sufficient geometric information. This holds true for our counts of curves. Therefore, we consider our tropical versions of stable maps as maps to $\RR^2$. The compactification is implicit in the choice of directions of the ends, resp.\ in the way they glue.

\begin{definition}\label{def-tropstablemap}
A \emph{tropical stable map} to $\RR^2$ is a tuple $(\Gamma,f)$ where $\Gamma$ is a (not necessarily connected) abstract tropical curve 
with $n$ marked ends denoted $x_1,\ldots,x_n$ and $f:\Gamma\rightarrow \RR^2$ is a map satisfying:
\begin{enumerate}
\item \emph{Integer affine on each edge:} On each edge $e$ of $\Gamma$, $f$ is of the form $$t\mapsto a+t\cdot v \mbox{ with } v\in \ZZ^2,$$ where we parametrize $e$ as an interval of size the length $l(e)$ of $e$. The vector $v$, called the \emph{direction}, arising in this equation is defined up to sign, depending on the starting vertex of the parametrization of the edge. We will sometimes speak of the direction of a flag $v(V,e)$. If $e$ is an end we use the notation $v(e)$ for the direction of its unique flag.
\item \emph{Balancing condition:} At every vertex, we have 
$$\sum_{e \in \partial V} v(V,e)=0.$$
\end{enumerate}
\end{definition}
 For an edge with direction $v=(v_1,v_2) \in \ZZ^2$, we call $w=\gcd(v_1,v_2)$ the \emph{expansion factor} and $\frac{1}{w}\cdot v$ the \emph{primitive direction} of $e$.

An isomorphism of tropical stable maps is an isomorphism of the underlying tropical curves respecting the map. 
The \emph{combinatorial type} of a tropical stable map is the data obtained when dropping the metric of the underlying graph. More explicitly, it consists of the data of a finite graph $\Gamma$, and for each edge $e$ of $\Gamma$, the direction of $e$.

\begin{definition}\label{def-tropcutmap}
A (cut, open) \emph{tropical stable map to $E_\TT\times\PP^1_\TT$} of \textit{leaky degree} $\Delta=\lbrace L_1,\dots,L_{d_2}\rbrace$ is a tropical stable map to $\RR^2$, with $N>n$ marked ends satisfying:
\begin{enumerate}
\item $\Delta$ is a multiset containing elements of $\ZZ$ such that $\sum_{i=1}^{d_2} L_i=0$.
\item  \emph{End directions:} The directions of the ends are given as follows:
\begin{itemize}
\item The marked ends $x_i$, $i=1,\ldots,n$, are contracted: $v(x_i)=0$.
\item There are $d_2$ ends of direction $(0,-1)$. All these ends are unmarked.
\item There are $d_2$ (unmarked) ends of direction $(L_i,1)$.  
\item The remaining ends are marked and of primitive direction $(\pm 1,0)$. 
\end{itemize}
\item \emph{Gluing:} The ends of primitive direction $(\pm 1,0)$ come in pairs, one of direction $(1,0)$ and one of $(-1,0)$, with the same expansion factor and the same $y$-coordinate.
 \end{enumerate}
When $L_i=0$ for all $i=1,\dots,d_2$, we refer to these tropical stable maps as maps without \textit{leaking}.
\end{definition}

\begin{construction}\label{const-gluing}
In two steps, we can produce a tropical stable map to $E_\TT\times\PP^1_\TT$ from a cut, open tropical stable map:
\begin{enumerate}
\item \label{const-gluing1}
We glue the pairs of ends in (2) of \ref{def-tropcutmap} and partially compactify by adding leaves at infinity for the ends of direction $(0,-1)$. The markings for the ends which we glued are forgotten.
The graph we obtain in this way is denoted $\Gamma'$. The map $f$ extends to $\Gamma'$ and the image $f(\Gamma')$ is an open tropical curve in $E_\TT\times\PP^1_\TT$. 
\item  \label{const-gluing2} To produce a tropical curve in $E_\TT\times\PP^1_\TT$, we need to mark the ends of directions $(L_i,1)$ and fix the following additional data:

\begin{enumerate}
\item A partition of $\Delta$ into subsets $\Delta_i=\lbrace L_{i_1},\dots,L_{i_{k_i}}\rbrace$ satisfying $\sum L_{i_j}=0$ for all $i$.
\item \label{tree} For each subset $\Delta_i$ a $3$-valent tree $T_i$ that satisfies the following.
	\begin{enumerate}
	\item
	The tree $T_i$ has $k_i+1$ leaves, where each $L_{i_j}$ for $j=1,\dots,k_i$ appears as a label of a leaf and exactly one leaf, called the \textit{root vertex}, is unlabeled.
	\item
	The tree $T_i$ is balanced in the following sense:
		Equip $T_i$ with the orientation that is induced by the root vertex such that the edge adjacent to the root vertex points towards the root vertex (notice that every non-leaf vertex of $T_i$ has precisely $2$ incoming edges and $1$ outgoing edge).
		Equip each edge adjacent to a leaf labeled by $L_{i_j}$ with the weight $L_{i_j}\in\mathbb{Z}$.
		For every non-leaf vertex $V$ of $T_i$, define the adjacent edges' weights by \textit{balancing}, i.e. the outgoing edge's weight of $V$ is the sum of the two incoming edges' weights. We do not allow a vertex $V$ of $T_i$ to have two incoming edges of the same weight .
	\end{enumerate}
\item
Each non-leaf vertex $V$ of $T_i$ is decorated with a number $n_V \in \mathbb{N}_{>0}$.
\end{enumerate}

With this additional data, we can produce a (compact) tropical curve in $E_\TT\times\PP^1_\TT$ from the open one: 
If two leaves $L_1$ and $L_2$ are adjacent to a vertex $V$ in $T_i$ (in particular, $L_1\neq L_2$ by condition (\ref{tree})), the two open ends of direction $(L_1,1)$ and $(L_2,1)$ meet in $E_\TT\times\PP^1_\TT$. At their $n_V$-th meeting point, we let them merge to a $3$-valent vertex and start a new open end at that vertex whose direction is determined by the balancing condition. We cut the cherry corresponding to $L_1$ and $L_2$ from the tree $T_i$ and continue recursively with the new tree.
Finally, we end up with the edge adjacent to the root vertex that produces an end of vertical direction $(0,k_i)$, which we compactify by adding vertices at infinity.

In this way, we produce a tropical curve in $E_\TT\times\PP^1_\TT$ whose upper vertical ends may have non-trivial expansion factors.

\end{enumerate}

Here, we focus on the open tropical curves we obtain using step (\ref{const-gluing1}) above. In particular, we will provide methods to count such open tropical curves and study their generating functions. We believe that our methods can be used in future research focusing on counts of curves in $E_\TT\times\PP^1_\TT$ with ends of non-trivial expansion factors, i.e.\ tropicalizations of curves in $E\times\PP^1$ satisfying tangency conditions with the $\infty$-section.

Note that for curves without leaking, step (\ref{const-gluing2}) of Construction \ref{const-gluing} is trivial: we only compactify by adding vertices at infinity for the vertical ends of direction $(0,1)$. We also neglect markings for these upper vertical ends. These curves, providing counts of curves satisfying point conditions in $E\times \mathbb{P}^1$, play the main role in our study.
\end{construction}

\begin{figure}[H]
\centering
\def\svgwidth{380pt}
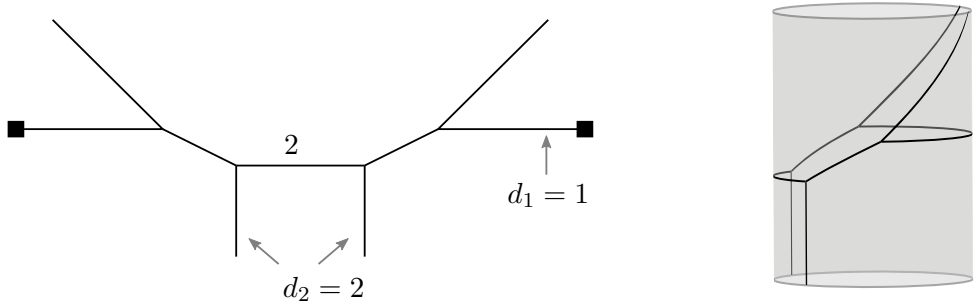
\caption{From left to right: The cut tropical stable map $f(\Gamma)$ to $E_\TT\times\PP^1_\TT$ of bidegree $(d_1,d_2)=(1,2)$ and genus $1$ and a non-cut picture, where we used Construction \ref{const-gluing} with two trivial trees $T_1,T_2$ consisting of two leaves each, and glued the ends with the black squares. The leaky degree is $\{-1,1\}$.}
\label{fig-glued1}
\end{figure}

\begin{figure}[H]
\centering
\def\svgwidth{280pt}
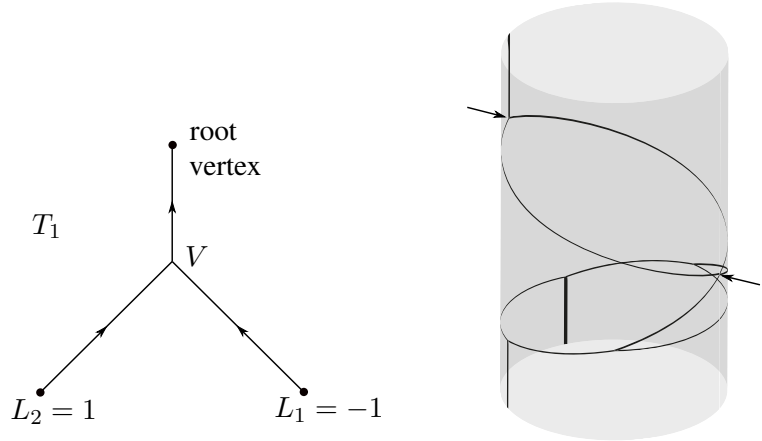
\caption{On the left, there is an example of a tree $T_1$ as in Construction \ref{const-gluing} with its orientation and labels. On the right, we see a glued picture that is produced by Construction \ref{const-gluing} by taking the left curve in Figure \ref{fig-glued1}, the tree $T_1$ and $n_V=2$. The two arrows mark the intersections which we need to take into account for $n_V=2$.}
\label{fig-glued2}
\end{figure}

Two tropical stable maps are \emph{equivalent}, if they differ only in the markings for the ends of primitive direction $(\pm 1,0)$, i.e.\ the glued graph and map to $E_\TT\times\PP^1_\TT$ from Construction \ref{const-gluing}(\ref{const-gluing1}) coincide. By abuse of notation, we will consider tropical stable maps only up to equivalence.

Depending on the image of their end vertices, an end of direction $(L_1,1)$ and an end of direction $(L_2,1)$ of a cut tropical stable map can intersect. In the dual subdivision for $f(\Gamma)$, such an intersection corresponds to a parallelogram. Consider the dual subdivision without these parallelograms, and let $d_1$ be the minimal distance of its vertices to its base line. Let $d_2$ be the number of ends of direction $(0,-1)$ which equals the number of ends of direction $(L_i,1)$ for all $i$. We call $(d_1,d_2)$ the \emph{bidegree} (see Figure \ref{fig-glued1}) of the tropical stable map $(\Gamma, f)$ for leaky degree $\Delta$ to $E_\TT\times\PP^1_\TT$.

We define the \emph{genus} of a tropical stable map $(\Gamma,f)$ to $E_\TT\times\PP^1_\TT$ to be the genus of the graph $\Gamma'$ obtained by gluing pairs of ends of $\Gamma$ as in Construction \ref{const-gluing}(\ref{const-gluing1}). We say that a tropical stable map to $E_\TT\times\PP^1_\TT$ is \emph{connected} if $\Gamma'$ is connected.
If $\nu\vdash d_1$ is the partition of expansion factors of the ends of $\Gamma$ of direction $(1,0)$ (equivalently, $(-1,0)$) and the genus of $(\Gamma,f)$ is $g$, then $\Gamma$ has genus $g-\ell(\nu)$.

For points $p_1,\ldots,p_n\in \RR^2$, a tropical stable map $(\Gamma,f)$ to
$E_\TT\times\PP^1_\TT$ \emph{satisfies the point conditions} $(p_1,\ldots,p_{n})$ if the point $f(x_i)$ to which the end marked $x_i$ is contracted equals $p_i$.

Notice that a tropical stable map to $E_\TT\times\PP^1_\TT$ satisfying point conditions can be viewed as a tropical curve to the plane satisfying end and point conditions, where the end conditions are imposed by gluing (see Definition \ref{def-troprelstablemap}). In particular, the theory of counts of plane tropical curves applies (see \cite{Mi03, GM051, GM052, GM053}).
The end conditions that are imposed by gluing are not necessarily in general position, for example there can be an edge of the glued graph $\Gamma'$ in Construction \ref{const-gluing}(\ref{const-gluing1}) that is ''curled'' several times. In $\Gamma$, this edge is cut into several connected components which are all mapped to the same horizontal line. If we shift the end conditions slightly, we have a tropical stable map to $\RR^2$ satisfying general conditions, and thus it is $3$-valent and, away from the contracted marked ends, locally an embedding.

\begin{example}\label{ex-dual}
Let $p_1,\dots,p_5\in\RR^2$ be some general positioned points. Figure \ref{fig-Pic5_ink} shows a tropical stable map to $E_\TT\times\PP^1_\TT$ of leaky degree $\Delta=\lbrace -1,1 \rbrace$ with a curled edge such that the point conditions are satisfied. Note that we need to fix pairs of ends with the same $y$-coordinate to make the gluing unique.

\begin{figure}[H]
\centering
\def\svgwidth{250pt}
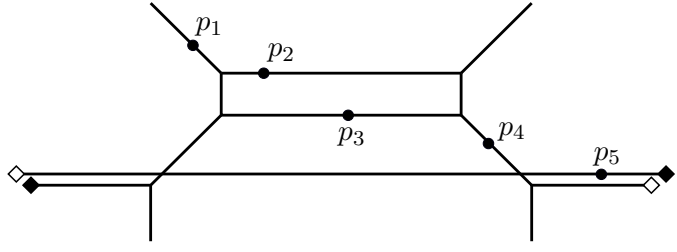
\caption{A (cut) tropical stable map satisfying point conditions. Note that one edge is curled once. The different pairs of end markings tell us how to glue, but note that we shifted the $y$-coordinates of glued ends a bit in order to get a better picture (in fact their $y$-coordinates are the same).}
\label{fig-Pic5_ink}
\end{figure}

\end{example}

For $n:=2d_2+g-1$ points $p_1,\ldots,p_{n}\in\RR^2$ in tropical general position, there are finitely many tropical stable maps $(\Gamma,f)$ to $E_\TT\times\PP^1_\TT$ of bidegree $(d_1,d_2)$ and genus $g$ satisfying the point conditions. Furthermore, each $\Gamma$ is $3$-valent, and $f$ is an embedding locally around the vertices which are not adjacent to contracted ends. In particular, we can define the multiplicity of a tropical stable map as in the original count of plane curves \cite{Mi03}:

\begin{definition}\label{def-mult} Let $(\Gamma,f)$ be a tropical stable map to $E_\TT\times\PP^1_\TT$, such that $\Gamma$ is $3$-valent, and $f$ is an embedding locally around the vertices which are not adjacent to marked ends. 
For a vertex $V$ of $\Gamma$ which is not adjacent to a marked contracted end, we define its \emph{vertex multiplicity} $\mult_V(\Gamma,f)$ to be $|\det(v_1,v_2)|$, where $v_1$ and $v_2$ denote the direction vectors of two of its adjacent edges.

We define the \emph{multiplicity} of $(\Gamma,f)$ to be the product of its vertex multiplicities:
$$\mult(\Gamma,f)=\prod_V \mult_V(\Gamma,f),$$ where the product goes over all vertices not adjacent to marked ends.
\end{definition}

\begin{definition}\label{def-tropGW}
Fix positive integers $d_1$, $d_2$ and $g$.
Fix $n=2d_2+g-1$ points $p_1,\ldots,p_n\in \RR^2$ in tropical general position.
Let $N^{\trop}_{(\Delta,d_1,d_2,g)}$ be the \emph{number of connected tropical stable maps to $E_\TT\times\PP^1_\TT$} of leaky degree $\Delta$, bidegree $(d_1,d_2)$ and genus $g$ satisfying the point conditions $p_1,\ldots,p_n$, counted with multiplicity as in Definition \ref{def-mult}. If we only list $3$ subscripts, this refers to the case without leaking: $N^{\trop}_{(d_1,d_2,g)}= N^{\trop}_{(\lbrace 0,\dots,0\rbrace,d_1,d_2,g)}$.
\end{definition}
As usual, the analogous count of not necessarily connected tropical stable maps is denoted by $N^{\trop,\bullet}_{(\Delta,d_1,d_2,g)}$ resp.\ $N^{\trop,\bullet}_{(d_1,d_2,g)}$.

\begin{definition}\label{def-troprelstablemap}
Fix partitions $\mu^+$, $\phi^+$, $\mu^-$ and $\phi^-$ such that the sum $d_1$ of the parts in $\mu^+$ and $\phi^+$ equals the sum of the parts in $\mu^-$ and $\phi^-$. Let $n_1=\ell(\phi^+)+\ell(\phi^-)$ and $n_2=\ell(\mu^+)+\ell(\mu^-)$. Fix  $n\in\NN_{>0}$ and a leaky degree $\Delta=\lbrace L_1,\dots,L_{d_2}\rbrace$ which is a multiset containing elements of $\ZZ$ such that $\sum_{i=1}^{d_2}L_i=0$.

A \emph{relative tropical stable map} matching the discrete data above is a (not necessarily connected) tropical stable map to $\RR^2$ of leaky degree $\Delta$, with $n+n_1+n_2$ marked ends satisfying:
\begin{enumerate}
\item The direction of the ends are imposed as follows:
\begin{itemize}
\item The marked ends $x_i$ for $i=1,\ldots,n$ are contracted: $v(x_i)=0$.
\item The other marked ends are of primitive direction $(\pm 1,0)$.
\item There are $d_2$ ends of direction $(0,-1)$.
\item There is an end of direction $(L_i,1)$ for each $i=1,\dots,d_2$.
\item The partition of expansion factors of the marked ends with primitive direction $(-1,0)$ is $(\mu^+,\phi^+)$.
\item The partition of expansion factors of the marked ends with primitive direction $(1,0)$ is $(\mu^-,\phi^-)$. 
\end{itemize}
 \end{enumerate}
\end{definition}
The \emph{genus} of a relative tropical stable map $(\Gamma,f)$ is defined to be the genus of $\Gamma$. We say $(\Gamma,f)$ is \emph{connected} if $\Gamma$ is.

For relative tropical stable maps, we can impose \emph{end conditions} by requiring that the $y$-coordinate of the horizontal line to which $f(x_i)$ is mapped equals a fixed value $y_i$ for the end marked $i$. By convention, we fix $y$-coordinates for the ends corresponding to $\phi^+$ and $\phi^-$.
For points and end conditions in general position, every relative tropical stable map satisfying the conditions has a $3$-valent source graph and is locally an embedding. We can define its multiplicity similar to Definition \ref{def-mult}:

\begin{definition}\label{def-multrel}
Let $\mu^+$, $\phi^+$, $\mu^-$ and $\phi^-$, $\Delta$, $d_1$, $d_2$, $n$, $n_1$ and $n_2$ be as in Definition \ref{def-troprelstablemap}.
Fix $n$ points in general position and $n_1$ $y$-coordinates for ends.
For a relative tropical stable map $(\Gamma,f)$ satisfying the points- and end conditions, we define its multiplicity as follows:
$$\mult(\Gamma,f)= \frac{1}{\prod \phi^+_i \prod_j \phi^-_j} \prod_V\mult_V(\Gamma,f),$$
where the product goes over all vertices $V$ not adjacent to contracted ends and $\mult_V(\Gamma,f)$ is defined in \ref{def-mult}.
\end{definition}

\begin{definition}\label{def-troprel} Fix $\mu^+$, $\phi^+$, $\mu^-$ and $\phi^-$, $\Delta$, $d_1$, $d_2$, $n$, $n_1$ and $n_2$ as in Definition \ref{def-troprelstablemap}, and a genus $g$, where we impose $ n=n_2+2d_2+g-1.  $
Fix $n$ points in general position and $n_1$ $y$-coordinates for ends.

For the invariant 
$\langle ({{\phi}}^-,{{\mu}}^-)| \tau_{0}(pt)^n|({{\phi}}^+,{{\mu}}^+)\rangle_{g,n}^{\trop,\Delta,(d_1,d_2)}$, we count connected \emph{relative tropical stable maps of genus $g$, matching the data and satisfying the conditions} with their multiplicity as defined in \ref{def-multrel}.
The invariant $\langle ({{\phi}}^-,{{\mu}}^-)| \tau_{0}(pt)^n|({{\phi}}^+,{{\mu}}^+)\rangle_{g,n}^{\trop,\Delta,(d_1,d_2),\bullet}$ denotes the analogous count of not necessarily connected relative tropical stable maps.
\end{definition}

\begin{proposition}\label{prop-tropdeg}
The number $N^{\trop,\bullet}_{(\Delta,d_1,d_2,g)}$ of tropical stable maps to $E_\TT\times\PP^1_\TT$ of leaky degree $\Delta$, bidegree $(d_1,d_2)$ and genus $g$ satisfying generic point conditions equals a sum of counts of relative tropical stable maps:

$$N^{\trop,\bullet}_{(\Delta,d_1,d_2,g)}= \sum_{(\mu,\phi) \;\vdash d_1} \frac{\prod_i \mu_i \prod_j\phi_j}{|\Aut(\mu)||\Aut(\phi)|}
\langle (\mu,\phi)|\tau_{0}(pt)^n |(\phi,\mu) \rangle_{g-\ell(\mu)-\ell(\phi), n}^{\trop,\Delta,(d_1,d_2),\bullet}.
$$
Here, the sum goes over all pairs of partitions $(\mu,\phi)$ of $d_1$.
\end{proposition}

\begin{proof}
Fix a tropical stable map $(\Gamma,f)$ to $E_\TT\times\PP^1_\TT$ contributing to $N^{\trop,\bullet}_{(\Delta,d_1,d_2,g)}$. 
Consider $\Gamma$ minus the closures of the contracted ends. Since the point conditions are in general position, each connected component contains at least one end. 
There can be connected components of $\Gamma$ which just consist of a single unbounded edge of primitive direction $(\pm 1,0)$, we will consider the left and the right part as an end. Such connected components arise from edges of the glued graph $\Gamma'$ from Construction \ref{const-gluing}(\ref{const-gluing1}) which are curled multiple times.
Let us first consider ends which are not part of such connected components.

An end is fixed by the point conditions if it is the unique end in its connected component. 
The other ends are moving: we can form a $1$-parameter family of tropical stable maps of the same combinatorial type that still meet the point conditions by shifting one of the moving ends slightly and letting the other edges follow. 

We treat one end of a component consisting of a single unbounded edge as a fixed end, and the other as a moving end, opposite to the assignment of the end they glue to.

Let $\phi$ be the partition of expansion factors of ends of primitive direction $(-1,0)$ which are fixed by the point conditions, and $\mu$ the partition of expansion factors of ends of primitive direction $(-1,0)$ which are moving.
Then $(\mu,\phi)$ is a partition of $d_1$. 
Furthermore, the gluing condition implies that the expansion factors of the ends of primitive direction are also given by $(\mu,\phi)$, however, the ones corresponding to $\phi$ must be moving, since the gluing already imposes a condition on them. Vice versa, the ones corresponding to $\mu$ must be fixed, since they impose conditions by gluing.

In this way, we can interpret each tropical stable map to $E_\TT\times\PP^1_\TT$ as a relative tropical stable map.

The factors of $\prod_i \mu_i \prod_j\phi_j$ show up because the multiplicity of the relative tropical stable maps differs from the multiplicity of the tropical stable map to $E_\TT\times\PP^1_\TT$ by factors of $\frac{1}{w}$ for each expansion factor $w$ of fixed ends, and the ends with expansion factors $\phi_j$ are fixed on the left while the ones with expansion factors $\mu_i$ are fixed on the right.

The factors of $\frac{1}{|\Aut(\mu)||\Aut(\phi)|}$ show up because the relative tropical stable maps have marked ends of primitive direction $(\pm 1,0)$, and we forget such markings for tropical stable maps to $E_\TT\times\PP^1_\TT$. 
\end{proof}

\subsection{Correspondence Theorems}

Correspondence Theorems for relative Gromov-Witten invariants of Hirzebruch surfaces have been studied before \cite{GM052,CJMR16}:
\begin{theorem}
Relative Gromov-Witten invariants of $\PP^1\times\PP^1$ are equal to their tropical counterparts. This holds both for the connected and the disconnected theory:
$$\langle (\mu,\phi)|\tau_{0}(pt)^n |(\phi,\mu) \rangle_{g, n}^{\PP^1\times\PP^1,(d_1,d_2)}=
\langle (\mu,\phi)|\tau_{0}(pt)^n |(\phi,\mu) \rangle_{g, n}^{\trop,\lbrace 0,\dots,0\rbrace,(d_1,d_2)}\mbox{ and }$$
$$
\langle (\mu,\phi)|\tau_{0}(pt)^n |(\phi,\mu) \rangle_{g, n}^{\PP^1\times\PP^1,(d_1,d_2),\bullet}= \langle (\mu,\phi)|\tau_{0}(pt)^n |(\phi,\mu) \rangle_{g, n}^{\trop,\lbrace 0,\dots,0\rbrace,(d_1,d_2),\bullet}.$$
\end{theorem}

Using the degeneration formula in Proposition \ref{prop-degeneration} together with the tropical relation of counts of stable maps to $E_\TT\times\PP^1_\TT$ and relative tropical stable maps, Proposition \ref{prop-tropdeg}, we can deduce:

\begin{theorem}[Correspondence Theorem]\label{thm-corres}
Gromov-Witten invariants of $E\times\PP^1$ agree with the corresponding tropical counts of stable maps to $E_\TT\times\PP^1_\TT$ without leaking:
$$N^\bullet_{(d_1,d_2,g)}= N^{\trop,\bullet}_{(d_1,d_2,g)}.$$
\end{theorem}
Since connectedness can be read off the dual graph of a degeneration, we can also deduce the version for connected numbers:
$$N_{(d_1,d_2,g)}= N^{\trop}_{(d_1,d_2,g)}.$$

\section{Pearl chains}\label{sec-pearl}
Using a floor diagram technique, we introduce a finite method to list all tropical stable maps of genus $g$, leaky degree $\Delta$ and bidegree $(d_1,d_2)$ to $E_\TT\times\PP^1_\TT$ --- we count \emph{curled pearl chains}.

\subsection{Pearl chains and curled pearl chains}
\begin{definition}\label{def-pearl}
Let $d_2$ and $g$ be positive integers. A \emph{pearl chain} of type $(d_2,g)$ is a (non-metric) connected graph $\mathcal{P}$ of genus $g$. It has $d_2$ white and $d_2+g-1$ black vertices. Edges can only connect a white with a black vertex, but not vertices of the same color. Black vertices must be $2$-valent, white vertices can have any valency. There are no cycles of length two.
\end{definition}

\begin{figure}[H]
\centering
\def\svgwidth{275pt}
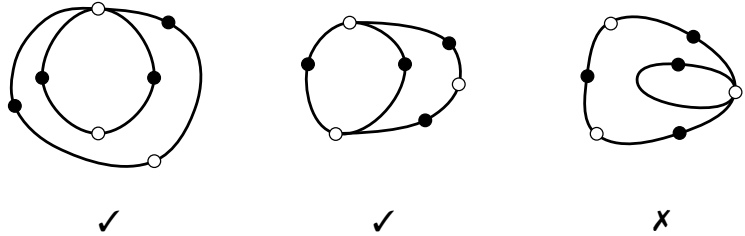
\caption{Some examples of pearl chains of type $(3,2)$ and a non-example (the right picture contains a cycle of length $2$).}
\label{fig-Pic6_ink}
\end{figure}

To define curled pearl chains, we have to define tropical covers first:

A \emph{tropical cover} $\pi:\Gamma_1\rightarrow \Gamma_2$ is a surjective map of metric graphs. The map $\pi$ is piecewise integer affine linear, the slope of $\pi$ on a flag or edge $e$ is a nonnegative integer called the \emph{expansion factor} $w_e\in \NN_{> 0}$.

For a point $v\in \Gamma_1$, and a flag $f'$ adjacent to $\pi(v)$, the \emph{local degree of $\pi$ at $v$ w.r.t.\ $f'$} is defined as the sum of the expansion factors of all flags $f$ adjacent to $v$ that map to $f'$:
\begin{equation}
d_{v,f'}=\sum_{f\mapsto f'} w_f.
\end{equation} 
We define the \textit{harmonicity} or \textit{balancing condition} at $v$ to be the fact that the local degree is independent of the choice of $f'$ (see \cite{ABBR1}, Section 2). 

Consider a tropical cover $\pi:\Gamma\rightarrow E_\TT$ to a tropical elliptic curve, and assume that the images of the vertices of $\Gamma$ are distinct. Let $p\in E_\TT$ be the image of the vertex $v$. For each other point in $\pi^{-1}(p)$, balancing is satisfied. Fix an orientation of $E_\TT$. Let $f_1$ and $f_2$ be the two flags of $E_\TT$ adjacent to $p$, ordered such that the orientation is respected. We say that there is a \emph{leaking} of $L\in\ZZ$ at $p$ (resp.\ at $v$) if $d_{v,f_1}-d_{v,f_2}=L$. If not all vertices of a tropical cover $\pi:\Gamma\rightarrow E_\TT$ satisfy balancing, but some have leaking, we say that it is a \emph{leaky tropical cover}.

For a leaky tropical cover $\pi:\Gamma\rightarrow E_\TT$, we let the \emph{degree} be the minimum of all sums over all local degrees of preimages of a point $a$ w.r.t.\ an adjacent flag $f'$, $d=\sum_{p\mapsto a} d_{p,f'}$.

\begin{definition}\label{def-curledpearl}
Fix positive integers $d_2,g$ and a multiset $\Delta=\lbrace L_1,\dots,L_{d_2} \rbrace$ containing elements of $\ZZ$ such that $\sum_{i=1}^{d_2} L_i=0$. Let $n=2d_2+g-1$. 
Let $E_\TT$ be a tropical elliptic curve on which we fix $n+1$ points $p_0,\ldots,p_n$ (notice that this choice fixes an orientation of $E_\TT$). Let $\mathcal{P}$ be a pearl chain of type $(d_2,g)$.
A \emph{curled pearl chain} of type $(\Delta,d_2,g)$ is a leaky tropical cover $\pi:\mathcal{P}'\rightarrow E_\TT$, where $\mathcal{P}'$ is a metrization of $\mathcal{P}$, such that each $\pi^{-1}(p_i)$ contains one vertex for $i=1,\ldots,n$, and such that 
\begin{itemize}
\item each element $L$ of $\Delta$ corresponds to a white vertex with leaking $L$ (in particular, there are $d_2$ white vertices), and
\item the black vertices are balanced.
\end{itemize}
\end{definition}

\begin{definition}\label{def-countpearl}
Let $N^{\pearl}_{(\Delta,d_1,d_2,g)}$ be the weighted \emph{count of curled pearl chains} of type $(\Delta,d_2,g)$ and degree $d_1$. Each curled pearl chain is counted with multiplicity $\prod_e w_e$, the product over all expansion factors of edges.
\end{definition}

\begin{example}\label{ex-countingpcs}
We want to determine $N^{\pearl}_{(\lbrace 0,\dots,0\rbrace,2,2,1)}$. We list curled pearl chains of type $(\lbrace 0,\dots,0\rbrace,2,1)$ and degree $2$ below, where we suppress $E_\TT$ and the map $\pi$ and fix the upper white vertex as preimage of $p_1$ instead (the numbers $i$ of the vertices in Figure \ref{fig-Pic7_ink} indicate to which point $p_i$ on $E_\TT$ the vertices are mapped to). One curled pearl chain has multiplicity $2^4=16$ and the rest has multiplicity $1$. So Figure \ref{fig-Pic7_ink} yields $16+14=30$ curled pearl chains counted with multiplicity. Since all vertices are $2$-valent, we can exchange the colors (i.e. fix a black vertex as preimage of $p_1$) and obtain a factor $2$. Therefore $N^{\pearl}_{(\lbrace 0,\dots,0\rbrace,2,2,1)}=60.$

\begin{figure}[]
\centering
\def\svgwidth{390pt}
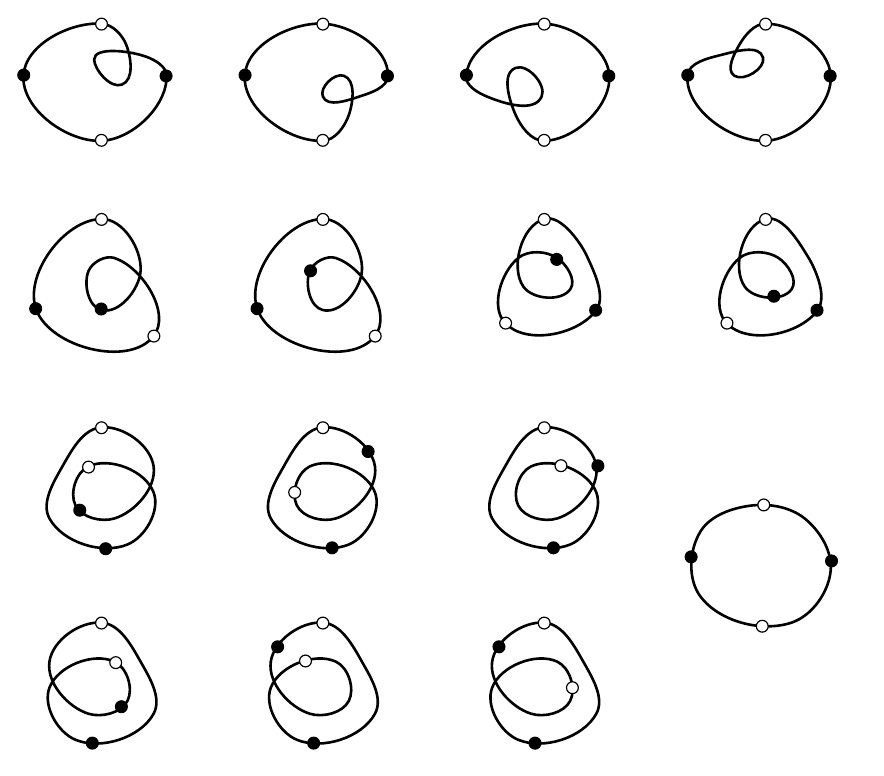
\caption{Curled pearl chains contributing to $N^{\pearl}_{(\lbrace 0,\dots,0\rbrace,2,2,1)}$.}
\label{fig-Pic7_ink}
\end{figure}

\end{example}

\subsection{Floor-decomposed tropical stable maps to $E_\TT\times\PP^1_\TT$}

Since the images of tropical stable maps to $E_\TT\times\PP^1_\TT$ can be viewed as tropical plane curves, we can also make use of the floor diagram technique 
\cite{BM08, FM09}: we use a horizontally stretched set of point conditions (see Definition 3.1 \cite{FM09}).

Every tropical stable map $(\Gamma,f)$ of leaky degree $\Delta=\lbrace L_1,\dots,L_{d_2}\rbrace$ satisfying these conditions is \emph{floor-decomposed}, i.e.\ every connected component of $\Gamma$ minus the edges of primitive direction $(\pm 1,0)$ (called a \emph{floor}) contains exactly one marked end, one end of direction $(0,-1)$ and one end of some direction $(L_i,1)$. Furthermore, each edge of primitive direction $(\pm 1,0)$ (called an \emph{elevator}), except for $n_1$ ends whose $y$-coordinates are imposed by the gluing condition, is adjacent to precisely one marked end.

\begin{construction}\label{const-pcfromtrop}
Given a floor-decomposed tropical stable map $(\Gamma,f)$ to $E_\TT\times\PP^1_\TT$ of bidegree $(d_1,d_2)$ and genus $g$ satisfying $n=2d_2+g-1$ horizontally stretched point conditions, we associate a curled pearl chain $\pi:\mathcal{P}'\rightarrow E_\TT$  as follows:

As in Construction \ref{const-gluing}(\ref{const-gluing1}), let $\Gamma'$ be the graph obtained from $\Gamma$ by gluing the pairs of ends of primitive direction $(\pm 1,0)$.
Shrink each floor in $\Gamma'$ to a white vertex. Shrink the marked points on elevators to black vertices. The graph obtained in this way is $\mathcal{P}'$. Its edges correspond to elevator edges of $\Gamma'$, see Figure \ref{fig-Pic8_ink}.

Let $\mbox{pr}$ denote the vertical projection of $E_\TT\times\PP^1_\TT$ to a tropical elliptic curve and let $p_1,\ldots,p_n$ be the images of the horizontally stretched point conditions. The map $\pi$ can be viewed as the composition of (the extension of) $f$ with $\mbox{pr}$.
It maps the white vertex arising from the floor containing $x_i$ to $p_i$ in $E$, and the black vertex arising from $x_j$ to $p_j$. The edge of $\mathcal{P}$ arising from an elevator edge of $\Gamma'$ is mapped as the projection of the elevator edge.
The expansion factors $w_e$ of the edges of $\mathcal{P}$ are given as the expansion factors of the corresponding elevator edges of $\Gamma'$.  Note that the leaky degree of the tropical stable map and the leaking of the associated curled pearl chain coincide.
\end{construction}

\begin{figure}[]
\centering
\def\svgwidth{380pt}
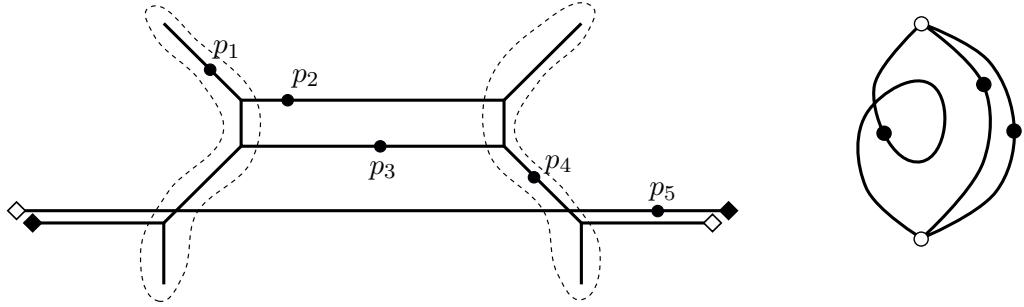
\caption{Left: The tropical stable map to $E_\TT\times\PP^1_\TT$ of leaky degree $\lbrace -1,1\rbrace$ with a curled edge from Example \ref{ex-dual}, where we indicated the floors. Right: The curled pearl chain construction \ref{const-pcfromtrop} associates to this curve (we suppressed the map $\pi$). }
\label{fig-Pic8_ink}
\end{figure}

\begin{lemma}\label{lem-pcfromtrop}
Construction \ref{const-pcfromtrop} associates a curled pearl chain $\pi:\mathcal{P}'\rightarrow E_\TT$ of type $(\Delta,d_2,g)$ and degree $d_1$ to a floor-decomposed connected tropical stable map $(\Gamma,f)$ to $E_\TT\times\PP^1_\TT$ of leaky degree $\Delta$, bidegree $(d_1,d_2)$ and genus $g$.
\end{lemma}

\begin{proof}
Denote $\Delta=\lbrace L_1,\dots,L_{d_2}\rbrace$. The tropical stable map has $d_2$ ends of direction $(0,-1)$ and exactly one end of directions $(L_i,1)$ for each $i=1,\dots,d_2$, and each floor contains one end of direction $(0,-1)$ and one of direction $(L_i,1)$ for some $i$. The white vertices of $\mathcal{P}'$ come from the $d_2$ floors. The black vertices come from the remaining $d_2+g-1$ points. Every elevator edge of $\Gamma'$ must be fixed by a point, so the corresponding edge in $\mathcal{P}$ must be adjacent to a black vertex. An edge cannot connect two black vertices, since this would correspond to an elevator edge adjacent to two contracted ends, which would then be mapped to the same horizontal line. Since the point conditions are general, two contracted ends cannot be mapped to the same horizontal line. It follows that each edge of $\mathcal{P}'$ is adjacent to one black and one white vertex. A black vertex comes from a contracted end and its two adjacent elevator edges, it must be $2$-valent. The graph $\mathcal{P}'$ arises from the glued graph $\Gamma'$ of Construction \ref{const-gluing}(\ref{const-gluing1}) by shrinking floors resp.\ ends, hence it has the same genus. Assume $\mathcal{P}'$ had a cycle of length $2$, which necessarily connects a white with a black vertex. This would correspond to an elevator edge of $\Gamma'$, starting at the floor corresponding to the white vertex and returning to the same floor with the same $y$-coordinate. The balancing condition implies that the floor cannot have a $3$-valent vertex adjacent to those two edges, the elevator loop would form a separate connected component, which we excluded. Thus $\mathcal{P}'$ has no cycles of length $2$. It follows that $\mathcal{P}$, which equals $\mathcal{P}'$ after forgetting metric data, is a pearl chain of type $(d_2,g)$.
From the construction of the map $\pi$, it is clear that the preimage of each $p_i$ contains exactly one vertex. 
Also, the two flags adjacent to a black vertex are mapped to the two flags in $E$ adjacent to the image vertex, and that their expansion factors agree: they correspond to the two elevator edges adjacent to a contracted end, which by the balancing condition must have direction $(w_e,0)$ resp.\ $(-w_e,0)$. 
For a white vertex, which represents a floor of $(\Gamma,f)$, note that it has two non-contracted ends of direction $(0,-1)$ and $(L_i,1)$ for some $i$. The other edges leaving the floor are the elevators, they are of primitive direction $(\pm 1,0)$. It follows from the balancing condition that there is leaking of $\Delta$ at the white vertices. Thus,  $\pi:\mathcal{P}'\rightarrow E_\TT$ is a curled pearl chain of type $(\Delta,d_2,g)$. For a floor-decomposed tropical stable map, the first entry of the bidegree $(d_1,d_2)$ equals the minimum of the sums of expansion factors of all elevators passing a fixed vertical line. This equals the degree of the curled pearl chain constructed from $(\Gamma,f)$.
\end{proof}

\subsection{Counts of tropical stable maps to $E_\TT\times\PP^1_\TT$ of leaky degree and pearl chains}

\begin{theorem}\label{them-troppearl}
The count of curled pearl chains of type $(\Delta,d_2,g)$ and degree $d_1$ from Definition \ref{def-countpearl} equals the number of connected tropical stable maps to $E_\TT\times\PP^1_\TT$ of leaky degree $\Delta$, genus $g$ and bidegree $(d_1,d_2)$ (see Definition \ref{def-tropcutmap}):
$$N^{\pearl}_{(\Delta,d_1,d_2,g)}= N^{\trop}_{(\Delta,d_1,d_2,g)}.$$
\end{theorem}

\begin{proof}
Given a tropical stable map contributing to $N^{\trop}_{(\Delta,d_1,d_2,g)} $, we know from Construction \ref{const-pcfromtrop} and Lemma \ref{lem-pcfromtrop} how to construct a curled pearl chain contributing to $N^{\pearl}_{(\Delta,d_1,d_2,g)}$. Here, we want to show that Construction \ref{const-pcfromtrop} yields a bijection between the set of tropical stable maps contributing to $N^{\trop}_{(\Delta,d_1,d_2,g)} $ and the set of curled pearl chains contributing to $N^{\pearl}_{(\Delta,d_1,d_2,g)}$, and that each tropical stable map counts with the same weight as the associated curled pearl chain. 

Pick generic horizontally stretched points $p_1',\ldots,p_n'$ in $\RR^2$.
Given a curled pearl chain $\pi:\mathcal{P}'\rightarrow E_\TT$, each vertex of $\mathcal{P}$ corresponds to a point $p_1,\dots,p_{n}\in E_\TT$ via its image under the map $\pi$.
We construct a tropical stable map $(\Gamma,f)$ satisfying the point conditions $p_1',\ldots,p_n'$ starting from local pieces of the image $f(\Gamma)\subset \RR^2$.

For a black vertex mapping to $p_i\in E$, we draw germs of elevator edges adjacent to $p_i'$. Each elevator edge of the glued graph $\Gamma'$ must be fixed by a unique point. We draw end germs for elevator edges of $\Gamma$ which are moving ends, but whose $y$-coordinate is imposed by the gluing conditions. If an edge of $\mathcal{P}$ curls multiple times, this corresponds to multiple connected components of $\Gamma$ consisting of a single edge which is mapped to a horizontal line. 

Consider a white vertex mapping to $p_j \in E$. It corresponds to a floor which satisfies the point condition $p_j'$. In $\Gamma$, a floor can be viewed as a path connecting an end of direction $(0,-1)$ with an end of direction $(L_i,1)$ --- the choice of $i$ here depends on the leaking of the corresponding white vertex. The edges adjacent to this path in $\Gamma$ are elevator edges, and they correspond to the edges of $\mathcal{P}$ adjacent to the white vertex. There is a unique way to connect the corresponding elevator edges (for which the horizontal line to which they map is already fixed) via a path, and to map this floor with $f$, such that it meets the point $p_j'$. In this way, we have constructed a floor-decomposed stable map $(\Gamma,f)$. This construction is obviously inverse to Construction \ref{const-pcfromtrop}, and so it follows in particular that the discrete data matches.

 The multiplicity of $(\Gamma,f)$ is given by the product of its local vertex multiplicities (see Definition \ref{def-mult}). Each $3$-valent vertex $V$ which is not adjacent to a contracted end is contained in a floor. Thus $V$ is adjacent to an elevator $e$.  Hence the multiplicity $\mult_V(\Gamma,f)$ of $V$ is given by
\begin{align*}
\mult_V(\Gamma,f)=\lvert \det
\left(
\begin{array}{cc}
w_e&*\\
0&\pm 1
\end{array}
\right)
\rvert = w_e,
\end{align*}
where the first column is the direction of the elevator $e$ and the second column is the direction of another edge adjacent to $V$. From the directions of the non-contracted ends that belong to the floor --- $(0,-1)$ and $(L_i,1)$ --- and the balancing condition, it follows that the $y$-coordinate of the direction vector must be $\pm 1$.
Thus every such vertex $V$ contributes a factor of the expansion factor of its adjacent elevator edge. Vice versa, every elevator edge is adjacent to one contracted end and one such vertex $V$. It follows that $\mult(\Gamma,f)$ equals the product of the expansion factors of its elevators, which equals the weight with which the associated curled pearl chain contributes to $N^{\pearl}_{(\Delta,d_1,d_2,g)}$ by Construction \ref{const-pcfromtrop} and Definition \ref{def-countpearl}.

\end{proof}

\section{Generating series and Feynman integrals}\label{sec-Feynman}
In this section, we study generating series of the numbers $N^{\trop}_{(\Delta,d_1,d_2,g)}$ in terms of Feynman integrals.

\subsection{Feynman integrals}

\begin{definition}[The propagator]\label{def-prop}
 We define the \emph{propagator} $$p(z,q):=\frac{1}{4\pi^2}\wp(z,q)+\frac{1}{12}E_2(q^2)$$ in terms of the Weierstra\ss{}-P-function $\wp$ and the Eisenstein series $$E_2(q):=1-24\sum_{d=1}^\infty \sigma(d)q^d.$$ Here, $\sigma=\sigma_1$ denotes the sum-of-divisors function $\sigma(d)=\sigma_1(d)=\sum_{m|d}m$.
\end{definition}

Changing coordinates $x=e^{i\pi z} $, the propagator has the following nice form (see Theorem 2.22 \cite{BBBM13}):
\begin{equation}
 P(x,q)= - \sum_{d=1}^\infty d\cdot x^{d} - \sum_{n=1}^\infty \left(\sum_{d|n}d \left(x^{d}+ x^{-d}\right)\right)q^{n}.\label{eq-prop}
\end{equation}
Here, we let $|x|<1$ and use a geometric series expansion (see Lemma 2.23 \cite{BBBM13}).

\begin{definition}[Feynman integrals] \label{def-Feynman} 
Let $\mathcal{P}$ be a pearl chain of type $(d_2,g)$. We fix a labeling $x_1,\ldots,x_n$ of its vertices and a labeling $q_1,\ldots,q_r$ of its edges.

Let $\Omega$ be a total order of the $n$ vertices of $\mathcal{P}$.

Denote the vertices adjacent to the edge $q_k$ by $x_{k^1}$ and $x_{k^2}$, where we assume $x_{k^1}<x_{k^2}$ in $\Omega$.

For integers $l_1,\ldots,l_n$ and using equation (\ref{eq-prop}), we define the \emph{Feynman integral} for $\mathcal{P}$ and $\Omega$ to be
$$I^{l_1,\ldots,l_n}_{\mathcal{P},\Omega}(q)=\Coef_{[x_1^{l_1}\ldots x_n^{l_n}]} \prod_{k=1}^{r} P(\frac{x_{k^1}}{x_{k^2}},q)$$
and the \emph{refined Feynman integral} to be
$$I^{l_1,\ldots,l_n}_{\mathcal{P},\Omega}(q_1,\ldots,q_r)=\Coef_{[x_1^{l_1}\ldots x_n^{l_n}]}\prod_{k=1}^{r} P(\frac{x_{k^1}}{x_{k^2}},q_k).$$
Finally, we set 
$$I^{l_1,\ldots,l_n}_{\mathcal{P}}(q)=\sum_\Omega I^{l_1,\ldots,l_n}_{\mathcal{P},\Omega}(q) $$
where the sum goes over all $n!$ orders of the vertices of $\Gamma$, and
$$I^{l_1,\ldots,l_n}_{\mathcal{P}}(q_1,\ldots,q_r)= \sum_\Omega I^{l_1,\ldots,l_n}_{\mathcal{P},\Omega}(q_1,\ldots,q_r).$$

If we drop the superscript $l_1,\ldots,l_n$ in the notations above, then this stands for $l_i=0$ for all $i$.
\end{definition}

For the case $l_i=0$ for all $i$, a Feynman integral can be viewed as a path integral in complex analysis:

\begin{definition}[Feynman integrals in complex analysis]\label{def-int}
Let $\mathcal{P}$ be a labeled pearl chain as in Definition \ref{def-Feynman}. Let $\Omega$ be an order.

Pick starting points of the form $iy_1,\ldots, iy_{n}$ in the complex plane, where the $y_j$ are pairwise different small real numbers.
We define integration paths $\gamma_1,\ldots,\gamma_{n}$ by $$\gamma_j:[0,1]\rightarrow \mathbb{C}:t\mapsto iy_j+t,$$ such that the order of the real coordinates $y_j$ of the starting points of the paths equals $\Omega$. 
We then define the integral \begin{equation}I'_{\mathcal{P},\Omega}(q):= \int_{z_j\in \gamma_j} \prod_{k=1}^{r} \left(-p(z_{k^1}-z_{k^2},q)\right),\label{eq-Igamma}\end{equation}
where $p(z,q)$ is the propagator of Definition \ref{def-prop},
and the refined version
\begin{equation}I'_{\mathcal{P},\Omega}(q_1,\ldots,q_r):= \int_{z_j\in \gamma_j} \prod_{k=1}^{r} \left(-p(z_{k^1}-z_{k^2},q_k)\right).\end{equation}
Here, as in Definition \ref{def-Feynman}, $x_{k^1}$ and $x_{k^2}$ are the two vertices adjacent to an edge $q_k$. 

Finally, we set
$$I'_{\mathcal{P}}(q)= \sum_{\Omega} I'_{\mathcal{P},\Omega}(q) \mbox{ and } I'_{\mathcal{P}}(q_1,\ldots,q_r)= \sum_{\Omega} I'_{\mathcal{P},\Omega}(q_1,\ldots,q_r).$$ 
\end{definition}

Since the propagator $p(z,q)$  is an even function, it is not important here in which way the vertices $x_{k^1}$ and $x_{k^2}$ of $q_k$ are ordered. The order $\Omega$ is only important for the arrangement of the integration paths.

\begin{theorem}\label{thm-Feynmanversions}

The complex analysis version of Feynman integrals agrees with the version from Definition \ref{def-Feynman}, using constant coefficients of formal series:

$$I_{\mathcal{P},\Omega}(q) = I'_{\mathcal{P},\Omega}(q)\mbox{ and }I_{\mathcal{P},\Omega}(q_1,\ldots,q_r) =I'_{\mathcal{P},\Omega}(q_1,\ldots,q_r).$$
\end{theorem}
For a proof, see Lemma 2.25 \cite{BBBM13}.

\subsection{Generating series and Feynman integrals}
We first state the main theorem and its corrolaries involving generating series of counts of curled pearl chains, tropical stable maps of leaky degree, and Gromov-Witten invariants.

\begin{theorem}\label{thm-genseries=Feynman}
Fix positive integers $d_2,g$ and a multiset $\Delta=\lbrace L_1,\dots,L_{d_2} \rbrace$ containing elements of $\ZZ$ such that $\sum_{i=1}^{d_2} L_i=0$.
The generating series of counts of curled pearl chains of type $(\Delta,d_2,g)$ (see Definition \ref{def-countpearl}) equals a sum of Feynman integrals:
$$\sum_{d_1} N^{\pearl}_{(\Delta,d_1,d_2,g)} q^{d_1}= \sum_{\mathcal{P}} \frac{1}{|\Aut(\mathcal{P})|}\sum_{v}\sum_{\Omega} I^v_{\mathcal{P},\Omega}(q).$$
The first sum on the right hand side goes over all pearl chains of type $(d_2,g)$ (see Definition \ref{def-pearl}), the second over all vectors $v$ which associates the leakings $L_1,\dots,L_{d_2}$ to the white vertices of $\mathcal{P}$, and $0$ to the black vertices.
\end{theorem}

We call the vectors $v$ over which we sum in Theorem \ref{thm-genseries=Feynman} the suitable \emph{leaking vectors}.

\begin{remark}\label{rem-summands} Notice that the generating series on the left of the equality in Theorem \ref{thm-genseries=Feynman} can be stratified into summands for pearl chains also --- counting only those curled pearl chains $\pi:\mathcal{P}'\rightarrow E_\TT$ for a fixed pearl chain $\mathcal{P}$. The equality holds for each summand indexed by a pearl chain $\mathcal{P}$.
\end{remark}

Using Theorem \ref{them-troppearl} and, for $\Delta=\lbrace 0,\dots,0\rbrace$, the Correspondence Theorem \ref{thm-corres}, we can interpret the generating series on the left as generating series of counts of tropical stable maps to $E_\TT\times\PP^1_\TT$ of some leaky degree $\Delta$, and, for $\Delta=\lbrace 0,\dots,0\rbrace$, also as the generating series of Gromov-Witten invariants of $E\times\PP^1$. In the latter case, we can also use the complex analysis version of Feynman integral.

\begin{corollary}\label{cor-genseries=Feynman,trop}
Fix positive integers $d_2,g$ and a leaky degree $\Delta$.
The generating series of counts of connected tropical stable maps to $E_\TT\times\PP^1_\TT$ of leaky degree, genus $g$ and bidegree $(d_1,d_2)$ (see Definition \ref{def-tropGW}) equals a sum of Feynman integrals:
$$\sum_{d_1} N^{\trop}_{(\Delta,d_1,d_2,g)} q^{d_1}= \sum_{\mathcal{P}} \frac{1}{|\Aut(\mathcal{P})|}\sum_{v} \sum_{\Omega} I^v_{\mathcal{P},\Omega}(q).$$
The first sum on the right hand side goes over all pearl chains of type $(d_2,g)$ (see Definition \ref{def-pearl}), the second over all suitable leaking vectors.
\end{corollary}

\begin{proof} This follows from Theorem \ref{thm-genseries=Feynman} and Theorem \ref{them-troppearl}, since the generating series on the left are equal.
\end{proof}

Also here, the left hand side can be stratified into sums corresponding to a fixed pearl chain, where we sum only over those tropical stable maps $(\Gamma,f)$ whose glued graph $\Gamma'$ (see Construction \ref{const-gluing}(\ref{const-gluing1})) equals $\mathcal{P}$ after shrinking floors and contracted ends and forgetting the metric.
We denote those counts by $N^{\trop,\mathcal{P}}_{(\Delta,d_1,d_2,g)}$. For each pearl chain $\mathcal{P}$, we have
$$\sum_{d_1} N^{\trop,\mathcal{P}}_{(\Delta,d_1,d_2,g)} q^{d_1}=  \frac{1}{|\Aut(\mathcal{P})|}\sum_{v} \sum_{\Omega} I^v_{\mathcal{P},\Omega}(q),$$
where the sum goes over all suitable leaking vectors.

\begin{corollary}\label{cor-genseries=Feynman,class}

Fix positive integers $d_2$ and $g$.
The generating series of Gromov-Witten invariants of $E\times\PP^1$ of genus $g$ and bidegree $(d_1,d_2)$ (see Definition \ref{def-GW}) equals a sum of Feynman integrals:
$$\sum_{d_1} N{(d_1,d_2,g)} q^{d_1}= \sum_{\mathcal{P}} \frac{1}{|\Aut(\mathcal{P})|}\sum_{\Omega} I'_{\mathcal{P},\Omega}(q).$$
The first sum on the right hand side goes over all pearl chains of type $(d_2,g)$ (see Definition \ref{def-pearl}.
\end{corollary}

\begin{proof}
This follows from Corollary \ref{cor-genseries=Feynman,trop}, the Correspondence Theorem \ref{thm-corres} and Theorem \ref{thm-Feynmanversions}, taking into account that the sum over all suitable leaking vectors $v$ in Corollary \ref{cor-genseries=Feynman,trop} becomes trivial for $\Delta=\lbrace 0,\dots,0 \rbrace$.
\end{proof}

Using tropical geometry, we can stratify the generating series on the left again into summands corresponding to a pearl chain: close to the tropical limit, we can take those stable maps which degenerate to a tropical stable map with a fixed underlying pearl chain. We denote these numbers by $N^{\mathcal{P}}{(d_1,d_2,g)}$, then we have
\begin{equation}\sum_{d_1} N^{\mathcal{P}}{(d_1,d_2,g)} q^{d_1}=  \frac{1}{|\Aut(\mathcal{P})|}\sum_{\Omega} I'_{\mathcal{P},\Omega}(q).\label{eq-summand}\end{equation}

\subsection{Labeled curled pearl chains and the proof of Theorem \ref{thm-genseries=Feynman}}

To prove Theorem \ref{thm-genseries=Feynman}, we need to introduce graph covers and labeled curled pearl chains:

\begin{definition}[Graph covers, see Definition 2.18 \cite{BGM18}] \label{def-graphcover}
Let $\Gamma$ be a graph with $n$ labeled vertices $x_1,\ldots,x_n$ and $r$ labeled edges $q_1,\ldots,q_r$. Let $\Omega$ be an order, viewed as an element in the symmetric group on $n$ elements, associating to $i$ the place $\Omega(i)$ that the vertex $x_i$ takes in the order $\Omega$. 
As before, fix an orientation of $E_{\TT}$ and points $p_0,p_1,\ldots,p_n$ ordered in this way when going around $E_{\TT}$ starting at $p_0$.

A \emph{graph cover} of type $\Gamma$ and order $\Omega$ is a leaky tropical cover $\pi:\Gamma'\rightarrow E_{\TT}$, where $\Gamma'$ is a metrization of $\Gamma$, such that $\pi^{-1}(p_{\Omega(i)})$ contains the vertex $x_i$. (Since there are $n$ point conditions and $n$ vertices, it follows that there is precisely one vertex of $\Gamma$ in each preimage $\pi^{-1}(p_{j})$).)

The \emph{multidegree} of a graph cover is the vector $\underline{a}\in \mathbb{N}^r$ with $k$-th entry $$a_k=|\pi^{-1}(p_0)\cap q_k|\cdot w_k,$$ where $w_k$ denotes the expansion factor of the edge $q_k$. 
 \end{definition}

\begin{definition}\label{def-labeledpearlchain}
A \emph{labeled curled pearl chain} $\pi:\mathcal{P}'\rightarrow E_\TT$ is a curled pearl chain (Definition \ref{def-curledpearl}) for which we fix labels $x_1,\ldots,x_n$ for the vertices of $\mathcal{P}$ and $q_1,\ldots,q_r$ for the edges.
\end{definition}

\begin{remark} By definition, a labeled curled pearl chain is a graph cover with a particular source graph, namely a pearl chain, and for which only suitable leaking vectors $v$ are allowed.
\end{remark}

\begin{definition}
Let $\mathcal{P}$ be a pearl chain of type $(d_2,g)$, $\Omega$ and order and $\underline{a}$ a multidegree. Let $v$ be a suitable leaking vector for $\Delta$.

We define $N^v_{\mathcal{P},\Omega,\underline{a}}$ to be the \emph{number of labeled curled pearl chains} whose source curve is a metrization of $\mathcal{P}$, of order $\Omega$ and multidegree $\underline{a}$, where the leaking is imposed by $v$. Each labeled curled pearl chain is counted with the product of the expansion factors $w_e$ of the edges $e$ of $\mathcal{P}$.
\end{definition}

Notice that $N^v_{\mathcal{P},\Omega,\underline{a}}$ is a count of graph covers with fixed order and multidegree as in Definition 2.18 \cite{BGM18}, for which we chose a particular source graph (namely a pearl chain) and a suitable leaking vector.

\begin{theorem}\label{thm-genseries=Feynmanfineversion}
Fix positive integers $d_2,g$ and a leaky degree $\Delta$. Fix a labeled pearl chain $\mathcal{P}$, an oder $\Omega$ and a multidegree $\underline{a}$.
Let $v$ be a suitable leaking vector.

The numbers $N^v_{\mathcal{P},\Omega,\underline{a}}$ of labeled curled pearl chains of Definition \ref{def-labeledpearlchain} are coefficients of a refined Feynman integral:
$$N^v_{\mathcal{P},\Omega,\underline{a}}= \Coef_{[q_1^{a_1}\ldots q_r^{a_r}]} I^v_{\mathcal{P},\Omega}(q_1,\ldots,q_r).$$
\end{theorem}

\begin{proof}
This follows from Theorem 2.19 \cite{BGM18}.
\end{proof}

The following is an immediate corollary of Theorem \ref{thm-genseries=Feynmanfineversion}:
\begin{corollary}\label{cor-fineseries}
Fix positive integers $d_2,g$ and a leaky degree $\Delta$. Fix a labeled pearl chain $\mathcal{P}$, an oder $\Omega$ and a multidegree $\underline{a}$.
Let $v$ be a suitable leaking vector.

The generating series of the numbers of labeled curled pearl chains (see Definition \ref{def-labeledpearlchain}) equals a refined Feynman integral:

$$\sum_{\underline{a}\in\NN^r} N^v_{\mathcal{P},\Omega,\underline{a}}
q_1^{a_1}\cdot \ldots\cdot q_r^{a_r} =  I^v_{\mathcal{P},\Omega}(q_1,\ldots,q_r).$$
\end{corollary}

\begin{proof}[Proof of Theorem \ref{thm-genseries=Feynman}]
The equality of the generating series and the sum of Feynman integrals now follows from Corollary \ref{cor-fineseries} by summing over all pearl chains $\mathcal{P}$ of type $(d_2,g)$ and leaking $\Delta$, and over all suitable leaking vectors $v$, setting the $q_k$ equal to $q$ again for all $k$ and keeping track of  automorphisms (as in the proof of Theorem 2.14 using Theorem 2.20 in \cite{BBBM13}). 
\end{proof}

\begin{example}\label{ex-pearl_chain}
Let $\underline{a}=(1,0,0,1)$ be a multidegree and let $v$ be zero, i.e. there is no leaking. Fix points $p_0,\dots,p_4$ on $E_\TT$ and let $\Omega$ be the order associated to the identity in the permutation group of $4$ elements. We want to determine $N^v_{\mathcal{P},\Omega,\underline{a}}$, where $\mathcal{P}$ and its labels are depicted on the left in the figure below. Using Figure \ref{fig-Pic7_ink} from Example \ref{ex-countingpcs}, we can see that the only labeled curled pearl chain of multidegree $(1,0,0,1)$ and order $\Omega$ is the one shown on the right below. The gray dots indicate the preimages of $p_0\in E_\TT$. Since this labeled curled pearl chain has multiplicity $1$, we have $N^v_{\mathcal{P},\Omega,\underline{a}}=1$.

\begin{figure}[H]
\centering
\def\svgwidth{330pt}
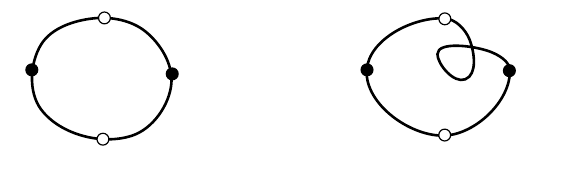
\label{fig-Pic9_ink}
\end{figure}

\noindent Using Theorem \ref{thm-genseries=Feynmanfineversion}, we can verify this count.
\end{example}

\begin{example}
We want to provide a lager example illustrating Theorem \ref{thm-genseries=Feynmanfineversion}. First of all, fix the following data: we consider no leaking, the pearl chain $\mathcal{P}$ (which is of genus $2$)  is shown below (see Figure \ref{fig-Pic10_ink}), the order $\Omega$ corresponds to the identical permutation and the total degree $d$ should be $3$. First, we want to determine the left hand side of Theorem \ref{thm-genseries=Feynmanfineversion} for the given input data, i.e. $\sum_{\underline{a}}N_{\mathcal{P},\Omega,\underline{a}}$, where the sum goes over all $\underline{a}$ contributing to a total degree of $3$. After that, we calculate the refined Feynman integral on the right hand side and compare its coefficients to the combinatorial count of labeled pearl chains from before.

\begin{figure}[H]
\centering
\def\svgwidth{280pt}
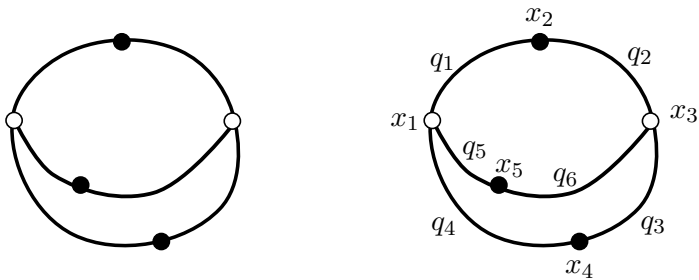
\caption{Left: a pearl chain $\mathcal{P}$. Right: $\mathcal{P}$ with labels.}
\label{fig-Pic10_ink}
\end{figure}

Since we fixed an order $\Omega$, we should label the pearl chain's vertices and edges (see Figure \ref{fig-Pic10_ink}) before curling them. In Figure \ref{fig-Pic11_ink} we determine $\sum_{\underline{a}}N_{\mathcal{P},\Omega,\underline{a}}$. Figure \ref{fig-Pic11_ink} shows schematic representations of labeled curled pearl chains: we suppress the elliptic curve $E_\TT$ and only show the source curve of each cover. The labels of the source curve are also suppressed, but can easily be seen from comparing Figure \ref{fig-Pic11_ink} to Figure \ref{fig-Pic10_ink}. Each curled pearl chain contributes with the product of its source curve's edge expansion factors (the edge expansion factors not equal to $1$ are shown in Figure \ref{fig-Pic11_ink}), thus $\sum_{\underline{a}}N_{\mathcal{P},\Omega,\underline{a}}=96$.

\begin{figure}[]
\centering
\def\svgwidth{380pt}
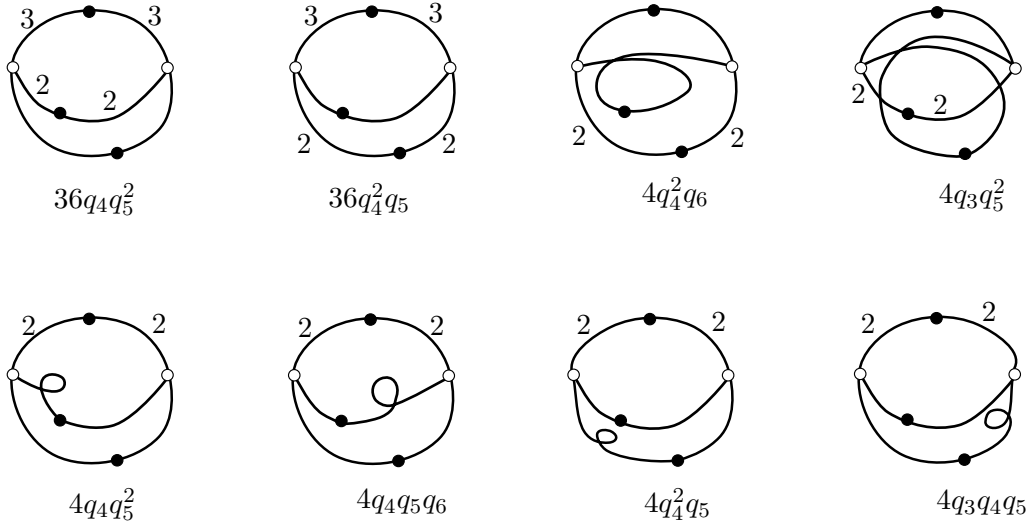
\caption{Curled pearl chains and their contributions to the refined Feynman integral.}
\label{fig-Pic11_ink}
\end{figure}

The sum of the refined Feynman integral's terms that are of degree $3$ is $40q_4^2q_5+40q_4q_5^2+4q_4q_5q_6+4q_3q_4q_5+4q_3q^2_5+4q^2_4q_6$, which makes $96q$ if we choose $q_1=\dots=q_4=q$, i.e. we obtained exactly the coefficient which we expected from our combinatorial count. The contributions of each term can directly be seen from our curled pearl chains. We wrote these contributions under each curled pearl chain in Figure \ref{fig-Pic11_ink}.

\end{example}

\section{Quasimodularity}\label{sec-quasimod}
Following the study of quasimodularity for generating series of Gromov-Witten invariants of an elliptic curve, Oberdieck and Pixton conjecture a quasimodularity statement for cycle-valued generating series of Gromov-Witten invariants of elliptic fibrations (Conjecture A \cite{OP17}), which they prove for the case of products $E\times B$ (Corollary 2). Quasimodularity behavior is useful because the asymptotics of the generating function becomes controllable. Quasimodularity of generating series of covers is a phenomenon studied beyond the case mentioned before, other important cases are generating functions of pillowcase covers \cite{EO06} or generating functions of numbers of covers of an elliptic curve with fixed ramifications with respect to the parity of the pullback of the trivial theta characteristic \cite{EOP08}.

 A series in $q$ is \emph{quasimodular} if and only if it is in the polynomial ring generated by the three \emph{Eisenstein series} $E_2$, $E_4$ and $E_6$ \cite{KZ95}. The \emph{weight} of a quasimodular form refers to its degree when viewed as a polynomial in the Eisenstein series. A series is called a quasimodular form of weight $w$ if it is a homogeneous polynomial of degree $w$ in the Eisenstein series, and it is called a quasimodular form of mixed weight if it is a non-homogeneous polynomial in the Eisenstein series.

Our Equation (\ref{eq-mainresult}) (see also Theorem \ref{thm-genseries=Feynman}, Corollary \ref{cor-genseries=Feynman,trop} and Corollary \ref{cor-genseries=Feynman,class}) hands us a new way to study quasimodularity by making use of Feynman integrals: the quasimodularity (of mixed weight) of a summand $I_{\mathcal{P}}$ for a fixed graph $\mathcal{P}$ can be deduced from the recent study of quasimodularity for graph sums in \cite{GM16}.

When passing to the tropical limit, the summands $I_{\mathcal{P}}$ obtain a meaning as summands of the generating series corresponding to a fixed pearl chain:

\begin{remark}\label{rem-seriesforonepc}
Fix positive integers $d_2,g$, and a pearl chain $\mathcal{P}$ of type $(d_2,g)$.

Then the following generating series coincide:
\begin{itemize}
\item the generating series of Gromov-Witten invariants of $E\times\PP^1$ of bidegree $(d_1,d_2)$ and genus $g$ which correspond to a fixed pearl chain $\mathcal{P}$ close to the tropical limit,
\item the generating series of tropical stable maps to $E_\TT\times \PP^1_\TT$ of bidegree $(d_1,d_2)$ and genus $g$ with fixed underlying pearl chain $\mathcal{P}$,
\item the generating series of curled pearl chains of type $(d_2,g)$ without leaking with fixed underlying pearl chain $\mathcal{P}$:
\end{itemize}

\begin{equation}
\sum_{d_1} N^{\mathcal{P}}_{(d_1,d_2,g)} q^{d_1}=  \sum_{d_1} N^{\trop,\mathcal{P}}_{(d_1,d_2,g)} q^{d_1} = \sum_{d_1} N^{\pearl,\mathcal{P}}_{(d_1,d_2,g)} q^{d_1}
\label{eq-equalseries}
\end{equation}

By Section \ref{sec-Feynman}, the series in Equation (\ref{eq-equalseries}) equal a sum of Feynman integrals
$$\frac{1}{|\Aut(\mathcal{P})|}\sum_\Omega I'_{\mathcal{P},\Omega}(q)=  \frac{1}{|\Aut(\mathcal{P})|}\sum_\Omega I_{\mathcal{P},\Omega}(q).$$
\end{remark}

\begin{theorem}\label{thm-quasimod}
Fix positive integers $d_2,g$ and a pearl chain $\mathcal{P}$ of type $(d_2,g)$. Each summand $I_{\mathcal{P},\Omega}(q)$ corresponding to an order $\Omega$ in the generating series (\ref{eq-equalseries}) is a quasimodular forms of mixed weight at most $4(d_2+g-1)$.
\end{theorem}
This follows from Theorem 1.1 resp.\ Theorem 6.1 in \cite{GM16}. These results imply that each summand $I_{\mathcal{P},\Omega}(q)$ corresponding to an order $\Omega$ is a quasimodular form of mixed weight at most $2r$, where $r$ is the number of edges of $\mathcal{P}$. Note that $r$ equals two times the number of back vertices of $\mathcal{P}$ (which is by definition $d_2+g-1$) because black vertices are $2$-valent and no black vertex is adjacent to a black vertex.

As a consequence, the generating series which is the sum of the ones in Equation (\ref{eq-equalseries}), summing over all pearl chains $\mathcal{P}$, is a quasimodular form:

\begin{corollary}\label{cor-quasimodular}
Let $d_2,g$ be positive integers.
Then the generating series of Gromov-Witten invariants of $E\times\PP^1$ (resp.\ of tropical stable maps to $E_\TT\times \PP^1_\TT$ of bidegree $(d_1,d_2)$ and genus $g$, resp.\ of curled pearl chains of type $(d_2,g)$) 

$$ \sum_{d_1} N_{(d_1,d_2,g)} q^{d_1}=\sum_{d_1} N^{\trop}_{(d_1,d_2,g)} q^{d_1}=\sum_{d_1} N^{\pearl}_{(d_1,d_2,g)} q^{d_1} $$
is a quasimodular form of mixed weight at most $4(d_2+g-1)$.
\end{corollary}

\begin{proof} This follows from the quasimodularity of each summand corresponding to a pearl chain $\mathcal{P}$ and an order $\Omega$ as considered in Theorem \ref{thm-quasimod} above. The maximal weight arises because a pearl chain of type $(d_2,g)$ has $2(d_2+g-1)$ edges, where $d_2+g-1$ is the number of black vertices.
\end{proof}

In the following, we give an example illustrating Corollary \ref{cor-quasimodular} which states that the series $\sum_{d_1} N^{\pearl}_{(d_1,d_2,g)} q^{d_1}$ is a quasimodular form. Recall that $\sum_{d_1} N^{\pearl}_{(d_1,d_2,g)} q^{d_1}$ is stratified as a sum over orders $\Omega$. 
We observe that for some input data each of these strata contributes with a summand that is a quasimodular form of mixed weight, but the sum over all orders yields a homogeneous quasimodular form.

This is in accordance with the situation of generating series of Hurwitz numbers, where each summand is quasimodular of mixed weight, but the whole sum is homogeneous \cite{GM16}.

\begin{example}\label{ex-homogeneous_forms}

We choose $d_2=2$ and $g=1$. Hence there is only one pearl chain $\mathcal{P}$ contributing to $\sum_{d_1} N^{\pearl}_{(d_1,d_2,g)} q^{d_1}$, namely the one shown in Example \ref{ex-pearl_chain}. Using Corollary~\ref{cor-fineseries}, we can calculate the generating series for our pearl chain $\mathcal{P}$ and any order $\Omega$ in terms of refined Feynman integrals. We observe that there are $8$ orders $\Omega$ (the elements of the symmetry group $D_4=\{ (), (1,2,3,4), (1,3)(2,4), (1,4,3,2), (2,4), (1,3), (1,2)(3,4),  (1,4)(2,3) \}$ of the square) which all lead to the same generating series
\begin{align*}
g_1:=\dots & 15092q^{11}+13560q^{10}+7701q^9+5680q^8+2520q^7+1872q^6+670q^5\\
	&+344q^4+92q^3+20q^2+q,
\end{align*}
which is a quasimodular form of mixed weight since it can be expressed in Eisenstein series as
\begin{align*}
g_1= &-\frac{1}{1080}E_6(q)+\frac{1}{1080}E_2(q)E_4(q)+\frac{1}{6912}E^2_4(q)-\frac{1}{2592}E_2(q)E_6(q)\\
	&+\frac{1}{3456}E^2_2(q)E_4(q)-\frac{1}{20736}E^4_2(q).
\end{align*}
Moreover, we observe that the remaining $16$ orders all lead to the same generating series
\begin{align*}
g_2:=\dots & 440q^{11}+2220q^{10}+888q^9+1000q^8+112q^7+360q^6+40q^5+52q^4\\ 
	&+8q^3+2q^2,
\end{align*}
which is also a quasimodular form of mixed weight since it can be expressed in Eisenstein series as
\begin{align*}
g_2= & \frac{1}{2160}E_6(q)-\frac{1}{2160}E_2(q)E_4(q)+\frac{1}{6912}E^2_4(q)-\frac{1}{2592}E_2(q)E_6(q)\\
	&+\frac{1}{3456}E^2_2(q)E_4(q)-\frac{1}{20736}E^4_2(q).
\end{align*}
Taking the sum, we obtain that the generating series
\begin{align*}
\sum_{d_1} N^{\pearl}_{(d_1,2,1)} q^{d_1} &=8g_1+16g_2 \\
&= \frac{1}{288}E^2_4(q)-\frac{1}{108}E_2(q)E_6(q)+\frac{1}{144}E^2_2(q)E_4(q)-\frac{1}{864}E^4_2(q)
\end{align*}
which is a homogeneous quasimodular form of weight $8$.
\end{example}

The observations from Example \ref{ex-homogeneous_forms} are true in general, which follows from the study of Oberdieck and Pixton (Theorem 7, Appendix A, \cite{OP17}):

\begin{theorem}\label{thm-quasimod homo}
Fix positive integers $d_2,g$ and a pearl chain $\mathcal{P}$ of type $(d_2,g)$. Let $r$ be the number of edges of $\mathcal{P}$. Then the generating series $\sum_{d_1} N^{\trop,\mathcal{P}}_{(d_1,d_2,g)} q^{d_1}$ obtained for a fixed pearl chain as the sum of the contributions over all orders is a homogeneous quasimodular form of weight $4(d_2+g-1)$.
\end{theorem}

\bibliographystyle {plain}
\bibliography {bibliographie}

\begin{thebibliography}{10}

\bibitem{ABBR1}
Omid Amini, Matthew Baker, Erwan Brugall{\'e}, and Joseph Rabinoff.
\newblock {Lifting harmonic morphisms I: metrized complexes and Berkovich
  skeleta}.
\newblock {\em Research in the Mathematical Sciences}, 2(1):52, 2015.

\bibitem{AB13}
Federico Ardila and Florian Block.
\newblock Universal polynomials for {S}everi degrees of toric surfaces.
\newblock {\em Adv. Math.}, 237:165--193, 2013.

\bibitem{AB14}
Federico Ardila and Erwan Brugall\'e.
\newblock The double {G}romov-{W}itten invariants of {H}irzebruch surfaces are
  piecewise polynomial.
\newblock {\em Int. Math. Res. Not. IMRN}, (2):614--641, 2017.

\bibitem{Beh97}
Kai Behrend.
\newblock Gromov-{W}itten invariants in algebraic geometry.
\newblock {\em Invent. Math.}, 127(3):601--617, 1997.

\bibitem{BF97}
Kai Behrend and Barbara Fantechi.
\newblock The intrinsic normal cone.
\newblock {\em Invent. Math.}, 128:45--88, 1997.

\bibitem{BBM10}
Beno\^{\i}t Bertrand, Erwan Brugall\'{e}, and Grigory Mikhalkin.
\newblock {Tropical Open Hurwitz numbers}.
\newblock {\em Rend. Semin. Mat. Univ. Padova}, 125:157--171, 2011.

\bibitem{BBM11}
Beno\^{\i}t Bertrand, Erwan Brugall\'{e}, and Grigory Mikhalkin.
\newblock Genus 0 characteristic numbers of tropical projective plane.
\newblock {\em Compos. Math.}, 150(1):46--104, 2014.
\newblock arXiv:1105.2004.

\bibitem{BG14b}
Florian Block and Lothar G\"{o}ttsche.
\newblock Fock spaces and refined {S}everi degrees.
\newblock {\em Int. Math. Res. Not. IMRN}, (21):6553--6580, 2016.

\bibitem{BG14}
Florian Block and Lothar G{\"o}ttsche.
\newblock Refined curve counting with tropical geometry.
\newblock {\em Compos. Math.}, 152(1):115--151, 2016.

\bibitem{BBBM13}
Janko B\"ohm, Kathrin Bringmann, Arne Buchholz, and Hannah Markwig.
\newblock Tropical mirror symmetry for elliptic curves.
\newblock {\em J. Reine Angew. Math.}, 732:211--246, 2017.

\bibitem{BGM18}
Janko B\"ohm, Christoph Goldner, and Hannah Markwig.
\newblock Tropical mirror symmetry in dimension $1$.
\newblock {Preprint}, arXiv:1809.10659, 2018.

\bibitem{Bexp}
Erwan Brugall\'e.
\newblock Floor diagrams relative to a conic, and {GW}-{W} invariants of del
  {P}ezzo surfaces.
\newblock {\em Adv. Math.}, 279:438--500, 2015.

\bibitem{BM:Pn}
Erwan Brugall\'e and Grigory Mikhalkin.
\newblock Enumeration of curves via floor diagrams.
\newblock {\em C. R. Math. Acad. Sci. Paris}, 345(6):329--334, 2007.

\bibitem{BM08}
Erwan Brugall\'{e} and Grigory Mikhalkin.
\newblock Floor decompositions of tropical curves: the planar case.
\newblock {\em Proceedings of the 15th G\"okova Geometry-Topology Conference},
  pages 64--90, 2008.
\newblock arXiv:0812.3354.

\bibitem{CJM10}
Renzo Cavalieri, Paul Johnson, and Hannah Markwig.
\newblock {Tropical Hurwitz numbers}.
\newblock {\em J. Algebr. Comb.}, 32(2):241--265, 2010.
\newblock arXiv:0804.0579.

\bibitem{CJMR17}
Renzo Cavalieri, Paul Johnson, Hannah Markwig, and Dhruv Ranganathan.
\newblock {Counting curves on toric surfaces: tropical geometry and the Fock
  space}.
\newblock {Preprint}, arXiv:1706.05401.

\bibitem{CJMR16}
Renzo Cavalieri, Paul Johnson, Hannah Markwig, and Dhruv Ranganathan.
\newblock {A graphical interface for the {G}romov--{W}itten theory of curves}.
\newblock In {\em Algebraic geometry: {S}alt {L}ake {C}ity 2015}, volume~97 of
  {\em Proc. Sympos. Pure Math.}, pages 139--167. Amer. Math. Soc., Providence,
  RI, 2018.

\bibitem{CP12}
Yaim Cooper and Rahul Pandharipande.
\newblock A {F}ock space approach to {S}everi degrees.
\newblock {\em Proc. Lond. Math. Soc. (3)}, 114(3):476--494, 2017.

\bibitem{Dij95}
Robbert Dijkgraaf.
\newblock {\em The moduli space of curves}, volume 129 of {\em Progr. Math.},
  chapter Mirror symmetry and elliptic curves, pages 149--163.
\newblock Birkh\"auser Boston, 1995.

\bibitem{EO06}
Alex Eskin and Andrei Okounkov.
\newblock Pillowcases and quasimodular forms.
\newblock In {\em Algebraic geometry and number theory}, volume 253 of {\em
  Progr. Math.}, pages 1--25. Birkh\"auser Boston, Boston, MA, 2006.

\bibitem{EOP08}
Alex Eskin, Andrei Okounkov, and Rahul Pandharipande.
\newblock The theta characteristic of a branched covering.
\newblock {\em Adv. Math.}, 217(3):873--888, 2008.

\bibitem{FM09}
Sergey Fomin and Grigory Mikhalkin.
\newblock Labeled floor diagrams for plane curves.
\newblock {\em J. Eur. Math. Soc.}, 12(6):1453--1496, 2010.
\newblock arXiv:0906.3828.

\bibitem{GM052}
Andreas Gathmann and Hannah Markwig.
\newblock The {Caporaso-Harris} formula and plane relative {Gromov-Witten}
  invariants in tropical geometry.
\newblock {\em Math. Ann.}, 338:845--868, 2007.
\newblock arXiv:math.AG/0504392.

\bibitem{GM051}
Andreas Gathmann and Hannah Markwig.
\newblock The numbers of tropical plane curves through points in general
  position.
\newblock {\em J. reine angew. Math.}, 602:155--177, 2007.
\newblock arXiv:math.AG/0504390.

\bibitem{GM053}
Andreas Gathmann and Hannah Markwig.
\newblock Kontsevich's formula and the {WDVV} equations in tropical geometry.
\newblock {\em Adv. Math.}, 217:537--560, 2008.
\newblock arXiv:math.AG/0509628.

\bibitem{GM16}
Elise Goujard and Martin M\"oller.
\newblock {Counting Feynman-like graphs: Quasimodularity and Siegel-Veech
  weight}.
\newblock arXiv:1609.01658, to appear in JEMS, 2016.

\bibitem{Gro09}
Mark Gross.
\newblock {Mirror Symmetry for $\mathbb{P}^2$ and tropical geometry}.
\newblock {\em Adv. Math.}, 224(1):169--245, 2010.
\newblock arXiv:0903.1378.

\bibitem{GS06}
Mark Gross and Bernd Siebert.
\newblock {Mirror Symmetry via logarithmic degeneration data I}.
\newblock {\em J. Differential Geom.}, 72:169--338, 2006.
\newblock arXiv:math.AG/0309070.

\bibitem{GS07}
Mark Gross and Bernd Siebert.
\newblock {Mirror Symmetry via logarithmic degeneration data II}.
\newblock {\em J. Algebraic Geom.}, 19(4):679--780, 2010.
\newblock arXiv:0709.2290.

\bibitem{KZ95}
Masanobu Kaneko and Don Zagier.
\newblock {\em The moduli space of curves}, volume 129 of {\em Progr. Math.},
  chapter A generalized Jacobi theta function and quasimodular forms., pages
  149--163.
\newblock Birkh\"auser Boston, 1995.

\bibitem{Li01A}
Jun Li.
\newblock Stable morphisms to singular schemes and relative stable morphisms.
\newblock {\em J. Diff. Geom.}, 57(3):509--578, 2001.

\bibitem{Li02}
Jun Li.
\newblock {A degeneration formula of GW-invariants}.
\newblock {\em J. Diff. Geom.}, 60:199--293, 2002.

\bibitem{Li11}
Si~Li.
\newblock {BCOV theory on the elliptic curve and higher genus mirror symmetry}.
\newblock Preprint, arXiv:1112.4063, 2011.

\bibitem{Lithesis}
Si~Li.
\newblock {\em Calabi-{Y}au {G}eometry and {H}igher {G}enus {M}irror
  {S}ymmetry}.
\newblock ProQuest LLC, Ann Arbor, MI, 2011.
\newblock Thesis (Ph.D.)--Harvard University.

\bibitem{MR16}
Travis Mandel and Helge Ruddat.
\newblock Descendant log {G}romov--{W}itten invariants for toric varieties and
  tropical curves.
\newblock arXiv:1612.02402, 2016.

\bibitem{Mi03}
Grigory Mikhalkin.
\newblock Enumerative tropical geometry in {${\mathbb{R}^2}$}.
\newblock {\em J. Amer. Math. Soc.}, 18:313--377, 2005.

\bibitem{NS06}
Takeo Nishinou and Bernd Siebert.
\newblock Toric degenerations of toric varieties and tropical curves.
\newblock {\em Duke Math. J.}, 135:1--51, 2006.
\newblock arXiv:math.AG/0409060.

\bibitem{OP17}
Georg Oberdieck and Aaron Pixton.
\newblock {Holomorphic anomaly equations and the {I}gusa cusp form conjecture}.
\newblock {\em Invent. Math.}, 213(2):507--587, 2018.

\bibitem{OP}
A.~Okounkov and R.~Pandharipande.
\newblock Gromov-{W}itten theory, {H}urwitz numbers, and matrix models.
\newblock In {\em Algebraic geometry---{S}eattle 2005. {P}art 1}, volume~80 of
  {\em Proceedings of Symposia in Pure Mathematics}, pages 325--414. American
  Mathematical Society, Providence, RI, 2009.

\bibitem{OP06}
Andrei Okounkov and Rahul Pandharipande.
\newblock {Gromov-Witten theory, Hurwitz theory, and completed cycles}.
\newblock {\em Ann. of Math.}, 163(2):517--560, 2006.

\bibitem{Ove15}
Peter Overholser.
\newblock Descendent tropical mirror symmetry for $\mathbb{P}^2$.
\newblock {Preprint, arXiv:1504.06138}, 2015.

\bibitem{Ran15}
Dhruv Ranganathan.
\newblock {Skeletons of stable maps {I}: Rational curves in toric varieties}.
\newblock {\em J. Lond. Math. Soc. (2)}, 95(3):804--832, 2017.

\bibitem{RS17}
Matteo Ruggiero and Kristin Shaw.
\newblock Tropical {H}opf manifolds and contracting germs.
\newblock {\em Manuscripta Math.}, 152(1-2):1--60, 2017.

\bibitem{Shaw15}
Kristin Shaw.
\newblock Tropical surfaces.
\newblock {Preprint, arXiv:1506.07407}, 2015.

\bibitem{SYZ}
Andrew Strominger, Shing-Tung Yau, and Eric Zaslow.
\newblock {Mirror symmetry is T-duality}.
\newblock {\em Nuclear Phys. B}, 479(1-2):243--259, 1996.

\bibitem{Vak08}
Ravi Vakil.
\newblock The moduli space of curves and {G}romov--{W}itten theory.
\newblock In {\em Enumerative invariants in algebraic geometry and string
  theory}, pages 143--198. Springer, 2008.

\end{thebibliography}

\end{document}